\documentclass{amsart}
\usepackage{amssymb}
\usepackage{amsmath}
\usepackage{amsfonts, latexsym}
\usepackage{color,graphics}
\usepackage{graphicx}
\usepackage{cite}
\usepackage{hyperref}
\usepackage{nicefrac}

\newcommand{\cpv}{-\!\!\!\!\!\!\!\!\;\int}

\newtheorem{theorem}{Theorem}[section]

\newtheorem{lemma}[theorem]{Lemma}

\newtheorem{definition}[theorem]{Definition}
\newtheorem{example}[theorem]{Example}
\newtheorem{remark}[theorem]{Remark}

\begin{document}

\author{Steven B. Damelin}
\address{Department of Mathematics, The University of Michigan, 2074 East Hall, 530 Church Street, Ann Arbor, MI, 48109, USA}
\email{damelin@umich.edu}

\author{Kai Diethelm} 
\address{Fakultät Angewandte Natur- und Geisteswissenschaften, Hochschule Würzburg-Schweinfurt, Ignaz-Schön-Str.\ 11,
97421 Schweinfurt, Germany}
\email{kai.diethelm@fhws.de}


\title[Weighted Singular Cauchy Integrals with Exponential Weights on $\mathbb R$]{An Analytic and Numerical Analysis of Weighted Singular Cauchy Integrals with Exponential Weights on $\mathbb R$}
\date{\today}

\begin{abstract}

This paper concerns an analytic and numerical analysis of a class of weighted singular Cauchy integrals with exponential weights $w:=\exp(-Q)$ with finite moments and with smooth external fields $Q:\mathbb R\to [0,\infty)$, with varying smooth convex rate of  increase for large argument. Our analysis relies in part on weighted  polynomial interpolation at the zeros of orthonormal  polynomials
with respect to $w^2$. We also study bounds for the first derivatives of a class of functions of the second kind for $w^2$.
\end{abstract}

\maketitle



\setcounter{equation}{0}
\section{Introduction}

Let $Q:\mathbb R\to [0,\infty)$ belong to a class of continuously differentiable functions  with varying smooth convex rate of increase for large argument. For a class of exponential weight functions $w := \exp(-Q)$ with finite moments, 
we investigate the weighted Cauchy principal value integral
\begin{equation*}
H_{w^2}[f;x] := \cpv_{\mathbb R} w^2(t) \frac{f(t)}{t-x} dt
          = \! \lim_{\varepsilon \to 0+} \!
            \left(
              \int_{-\infty}^{x-\varepsilon} \!w^2(t) \frac{f(t)}{t-x} dt + \! \int_{x+\varepsilon}^{\infty} \! w^2(t) \frac{f(t)}{t-x} dt
            \right)
\end{equation*}
with respect to its analytical properties, and we develop and analyze numerical methods for the approximate calculation of such integrals.

Here, we work on the real line, i.e.\ $x \in \mathbb R$ is arbitrary but fixed and $f:\mathbb R\to \mathbb R$ belongs to a class of functions for which in particular, $H_{w^2}[f;x]$ is finite. When we say $w$ has finite moments, we mean that $\int_{\mathbb R}x^nw^2(x)<\infty$, ($n=0,1,2,...$). 
Notice that the numerator in the integrand of the operator
$H_{w^2}[f,x]$ is $f w^2$.%

One reason for our  investigation  is due to the fact that integral equations with weighted Cauchy principal value integral kernels have shown to be an important tool for the modelling of many
physical situations.  See for example \cite{CDLM1, DD1, DD2, DD3, DD4, D, D1, K, M} and the references cited therein.
In the case of ordinary integrals (without strong singularities) on unbounded intervals, various interesting results can be found for example in  \cite{LM,SS}. 

In a series of papers, \cite{DD1,DD2,DD3,DD4}, the authors studied this problem and some of its applications for a class of weights $w:=\exp(-Q)$ with finite moments and with even external fields $Q:\mathbb R\to [0,\infty)$ belonging to a class of continuously differentiable functions  with  smooth polynomial rate of increase for large argument. An example of such an external field $Q$ studied is $|x|^{\alpha}$ where $\alpha>1$.

In this paper, we extend the results of \cite{DD1} to a class of exponential weight functions $w=\exp(-Q)$ with finite moments and with external fields $Q:\mathbb R\to [0,\infty)$  continuously differentiable with certain convex increase for large argument. In particular, this class of external fields $Q$ studied 
need not be even (a considerably weaker condition on $Q$) and may allow for considerable varying convex rates of increase for large argument for example smooth polynomial increase and also faster than smooth polynomial increase. Typical examples \cite{LL,Lub} of admissible external fields $Q$ would be with some $\beta \ge \alpha > 1$ and $\ell,k\geq 0$, 
\begin{equation*}
Q_{1,\alpha,\beta}(x):=
        \begin{cases}
	  x^{\alpha}, & x \in [0,\infty) \\
	  |x|^{\beta}, & x \in (-\infty,0)
        \end{cases} 
\end{equation*}
or
\begin{equation*}
Q_{2,\alpha,\beta,\ell,k}(x):=
        \begin{cases}
	  \exp_l(x^{\alpha})-\exp_l(0), & x \in [0,\infty) \\
	  \exp_k(|x|^{\beta})-\exp_k(0), & x \in (-\infty,0)
        \end{cases}
\end{equation*}

Here, for $x\in \mathbb R$, $\exp_{0}(x):=x$ and for $j\geq 1$, $\exp_j(x):=\exp(\exp(\exp...\exp(x)))$, $j$ times is the $j$th iterated exponential. In particular, $\exp_j(x)=\exp(\exp_{j-1}(x))$.\footnote{$\exp(-Q_{1,\alpha,\beta})$ and $\exp(-Q_{2,\alpha,\beta, \ell,k})$ are historically  often called respectively Freud and Erd\H{o}s weights. See \cite{LL,Lub} and the many references cited therein.}

\subsection{A note on notation and constants}

Throughout, $|.|$ is the Euclidean metric on $\mathbb R$. $\mathcal{P}_n$ will denote the class of polynomials of degree at most $n\geq 1$.  $\| \cdot \|_{L_{p}}$, $0< p\leq \infty$, will denote the usual $L_p$ function space norm. Sometimes we will write the shorthand form $\| \cdot \|_p$ when the context is clear.
For $\gamma > 0$, we identify the space $\mathrm{Lip}(\gamma)$ as the space of functions $f:\mathbb R\to \mathbb R$ for which $f$ is Lipschitz of order $\gamma$.
$C, C_1, C_2,...$ will denote positive constants independent of $n,x$ and may take on different values at different times.
The context will be clear. Function and operator notation (for example $f,H$) may also denote different or the same function/operator at different times. The context will be clear. When we write for a function, $f(\cdot)$, constant $C(\cdot)$ or operator $H(\cdot)$, we mean that the function/constant/operator depends on the indicated quantity. Dependence on several quantities follows a similar convention.
Finally, for non zero real sequences $\alpha_n$ and $\beta_n$, we write $\alpha_n=O(\beta_n)$ if there exists a constant $C>0$ so that uniformly in all other parameters that $\alpha_n$ and $\beta_n$ may depend on, $\frac{\alpha_n}{\beta_n}\leq C$, $\alpha_n=o(\beta_n)$ if $\frac{\alpha_n}{\beta_n}\to 0$, $n\to \infty$
and $\alpha_n\sim\beta_n$ if $\alpha_n=O(\beta_n)$ and $\beta_n=O(\alpha_n)$. Similar notation holds for sequences of functions or operators.

\setcounter{equation}{0}
\section{The class of weights, $a_t$, $a_{-t}$ and some further important quantities}

In this section, we introduce our class of weights and introduce some important quantities needed to move forward, including critical functions denoted by $a_t$ and $a_{-t}$ which we use throughout. 

\subsection{The class of admissible weights}

Motivated by the external fields $Q_{1,\alpha,\beta}$ and $Q_{2,\alpha,\beta,\ell,k}$ defined above, we define our class of weights \cite{LL, Lub}.  
To formulate our definition, we shall say that a function $f : \mathbb R\rightarrow [0,\infty)$
is quasi-increasing on $[0,\infty)$ if there exists $C > 0$ such that $f(x) \le C f(y)$ for all $x, y \in \mathbb R$ with $0 < x\leq  y<\infty$.
The concept quasi-decreasing is defined similarly.

Following is now our class of weights:

\begin{definition}\label{Definition 1.1}
  Let $Q: \mathbb R \rightarrow [0,\infty)$ satisfy the following properties:
  \begin{enumerate}
  \item[(a)] $Q'(x)$ exists and is continuous in $\mathbb R$, with $Q(0)=0$. Moreover, $Q''$ exists in $\mathbb R$.
  \item[(b)] $Q'(x)$ is non-decreasing in $\mathbb R$.
  \item[(c)]
    \[
      \lim_{x\rightarrow \infty} Q(x) = \lim_{x\to -\infty} Q(x) = \infty.
    \]
  \item[(d)] The function
    \begin{equation*}
      T(x) := \frac{xQ'(x)}{Q(x)}, \quad x\neq 0
    \end{equation*}
    is quasi-increasing in $(0,\infty)$, is quasi-decreasing in $(-\infty,0)$ and satisfies
    \begin{equation*}
      T(x) \ge C > 1, \quad  x \in \mathbb R\setminus \{ 0 \}.
    \end{equation*}
\item[(e)]  For all $x \in \mathbb R$, 
\[
Q''(x) Q(x) \le C_1 (Q'(x))^2
\]
 and there exists a compact subinterval $I$ of $\mathbb R$ and $C_2>0$ so that for a.e\, \, $x \in I\setminus \{ 0 \}$.
\[
Q''(x) Q(x) \geq C_2 (Q'(x))^2.
\]
 \item[(f)] There exists $\varepsilon_0\in(0,1)$ such that for $y\in \mathbb R\backslash\{0\}$,
    \begin{equation*}
      T(y) \sim T \left( y \left| 1 - \frac{\varepsilon_0}{T(y)} \right| \right).
    \end{equation*}
  \item[(g)] Assume that there exist $C$, $\varepsilon_1 > 0$ such that
    \begin{equation*}
      \int_{x-\frac{\varepsilon_1 |x|}{T(x)}}^x \frac{|Q'(s)-Q'(x)|}{|s-x|^{3/2}} ds
      \le C |Q'(x)| \sqrt{\frac{T(x)}{|x|}},  \quad x \in \mathbb R \setminus \{ 0 \}.
    \end{equation*}
  \end{enumerate}
  Then we say that $w=\exp(-Q)$ is an admissible weight with external field $Q$. 
\end{definition}

Let us illustrate Definition \ref{Definition 1.1}  using the examples of $Q_{1,\alpha,\beta}$ and $Q_{2,\alpha,\beta,\ell,k}$ in Section 1. 

\subsubsection{The Freud-type weight $Q_{1,\alpha,\beta}$}

Here, a straightforward calculation shows that $T\sim 1$ in $\mathbb R$. Thus (d) holds. Notice that $T=O(1)$ forces $Q_{1,\alpha,\beta}$ to be of smooth polynomial growth for large argument.  Indeed it is straightforward to show that $T(x)=\alpha$ for $x\in (0,\infty)$ and $T(x)=\beta$ for $x\in (-\infty,0)$. The conditions (a,b,c,e,f,g) are  straightforward to check.

\subsubsection{The Erd\H os-type weight $Q_{2,\alpha,\beta,\ell,k}$}

Here, a straightforward calculation shows that $T$ grows without bound for large argument. Thus (d) holds.  Notice that $T$ growing without bound for large argument forces $Q_{2, \alpha,\beta,k,l}$ to be of faster than smooth polynomial growth for large argument.
Indeed, it is straightforward to check that  if $l\geq 1$ and $x>0$, 
\[
T(x)=\alpha x^{\alpha}\left[\prod_{j=1}^{l-1}\exp_j(x^{\alpha})\right]\frac{\exp_l(x^{\alpha})}{\exp_l(x^{\alpha})-\exp_l(0)}.
\]
Indeed, $T(x)\to \alpha,\, x\to 0+$ while
\[
T(x)=\alpha x^{\alpha}\left[\prod_{j=1}^{l-1}\exp_j(x^{\alpha})\right](1+o(1)), \quad x\to \infty
\]
 with a similar expression holding for $x<0$. The conditions (a,b,c,e,f,g) are  straightforward to check.

\subsection{The quantities $a_t$ and $a_{-t}$}.

\begin{definition}[cf.\ \cite{Deift,LL,Lub,Mh}]
  Given an admissible weight $w$ and some $t > 0$, we define the quantities $a_t > 0$ and $a_{-t} < 0$
  as the unique solutions to the equations
  \begin{subequations} \label{1.1}
  \begin{eqnarray} 
    t & = & \frac{1}{\pi} \int_{a_{-t}}^{a_t} \frac{x Q'(x)}{\sqrt{(x-a_{-t})(a_t-x)}} dx, \\ 
    0 & = & \frac{1}{\pi} \int_{a_{-t}}^{a_t} \frac{Q'(x)}{\sqrt{(x-a_{-t})(a_t-x)}} dx.
  \end{eqnarray}
  \end{subequations}
\end{definition}

\begin{remark}
 In the special case where $Q$ is even, the uniqueness of $a_{\pm t}$ forces
 $a_{-t} = -a_t$ for all $t > 0$. In this case, $a_t$ is the unique positive solution of the equation
  \[
     t = \frac 2 \pi \int_0^1 \frac{a_t u Q'(a_t u)}{\sqrt{1-u^2}} du.
  \]
\end{remark}

Following \cite{LL}, we further use the notations
\[
  \Delta_t = [a_{-t}, a_t]
\]
and
\[
  \beta_t := \frac 1 2 (a_t + a_{-t}); \quad \delta_t := \frac 1 2 (a_t + |a_{-t}|)
\]
and define
\begin{equation*}
  \varphi_n(x) :=
  \begin{cases}
    \displaystyle
    \frac{|x-a_{-2n}| \cdot |a_{2n}-x|}{\sqrt{(|x-a_{-n}| + |a_{-n}| \eta_{-n})(|x-a_{n}| + |a_{n}| \eta_{n})}},
                       & x \in \Delta_n,\\
    \varphi_n(a_n), & x > a_n,\\
    \varphi_n(a_{-n}), & x < a_{-n}
  \end{cases}
\end{equation*}
where
\begin{equation}\label{eta}
  \eta_{\pm n} := \left( n T(a_{\pm n}) \sqrt{\frac{|a_{\pm n}|}{\delta_n}} \right)^{-2/3}.
\end{equation}

\begin{example}
\begin{itemize}
\item[(a)] Consider the weight $Q_{1,\alpha,\beta}$: Here we have for $n \to \infty$,
\[
T(a_{\pm n})\sim 1
\]  
and
\[
      a_n \sim n^{1/\alpha}, \quad |a_{-n}| \sim n^{(2\alpha-1)/(\alpha(2\beta-1))}.
    \]
\item[(b)] 
Consider the weight $Q_{2,\alpha,\beta,\ell,k}$: Here we have for $n \to \infty$
  \[
      T(a_n)\sim \prod_{j=1}^{\ell}\log_jn,
      \quad
      T(a_{-n})\sim \prod_{j=1}^{k}\log_jn
    \]
    and
    \[
      a_n    = (\log_{\ell}n)^{1/\alpha} (1+o(1)); \quad
      a_{-n} = (\log_{k}n)^{1/\beta} (1 + o(1))
    \]
    where we recall 
    \[
      \exp_{j}(x) = \underbrace{\exp(\exp(\exp\ldots\exp(x)))}_{j \,\,\textrm{times}}
    \]
    is the $j$th iterated exponential and for $x > 0$, $\log_0(x)=x$ and for $x>\exp_{j-1}(0)$ and $j\geq 1$,
    \[ 
      \log_j(x) = \underbrace{\log(\log(\log\ldots\log(x)))}_{j \,\,\textrm{times}},\,  
    \]
    is the $j$th iterated logarithm (and not the logarithm with respect to base $j$).
  \end{itemize}
\end{example}

The precise interpretations of the functions $a_{t}$ and $a_{-t}$, $t>0$ arise from logarithmic potential theory where they are in fact scaled endpoints of the support of a minimizer for a certain weighted variational problem in the complex plane.
(Hence the term external field for $Q$). When $Q$ is convex, the support of the minimizer is one interval and when $Q$ is in addition even,  $a_{-t}=-a_t$ holds for every $t>0$,
cf.\ \cite{Deift,LL,Lub,Mh}.  We shall use the important identity (and its $L_p$ cousins for different $0<p<\infty$)  
\[
	\|P_n w\|_{L_{\infty}(\mathbb R)} = \|P_n w\|_{L_{\infty}[a_{-n},a_n]}
\] 
valid for every polynomial $P\in \mathcal{P}_n$, $n \geq 1$. In the case when $w$ is even, $a_n$, $n\geq 1$, is asymptotically the smallest number for which this identity holds \cite{SD,LL,Lub,Mh}. 
This identity is useful in the sense that it can be used to get an intuitive idea of the growth of $a_{n}$ and $a_{-n}$ for large $n$ for different admissible weights $w$, for example for the admissible weights $w_{1,\alpha,\beta}$ and $w_{2,\alpha,\beta,\ell,k}$.

\setcounter{equation}{0}
\section{Main Result and Important Quantities}

In this section, we will state one of our main results, Theorem \ref{Theorem 2.3}.  
Moreover, we will also introduce and define various important quantities.
In the remaining sections, we will provide the proof of Theorem \ref{Theorem 2.3} 
and state and prove several more main results which are consequences of our machinery. 

Let $w$ be admissible. Then we can construct orthonormal polynomials $p_n(x)=p_n(w^2,x)$ of degree $n = 0, 1, 2, \ldots$ for $w^2(x)$ satisfying 
\begin{equation}
	\label{eq:onp}
   \int_{\mathbb R} p_n(x) p_m(x) w^2(x) dx = \delta_{mn}.
\end{equation}

Here for $m,n\geq 0$, $\delta_{mn}$ takes the value $1$ when $m=n$ and $0$ otherwise.

The zeros $x_{j,n}$, $j = 1, 2, \ldots, n$, 
of $p_n$ above will serve as the nodes of certain interpolatory quadrature formulas $Q_n$ which are therefore defined by
\begin{equation}\label{w_jn}
  Q_n[f;x] = \sum_{j=1}^n \zeta_{jn}(x) f(x_{j,n})
\end{equation}
and where the weights $\zeta_{jn}(.)$ are chosen such that the quadrature error $R_n$ satisfies
\[
  R_n[f;x] := H_{w^2}[f;x] - Q_n[f;x] = 0
\]
for every $x\in \mathbb R$ and every $f \in \mathcal{P}_{n-1}$.
In other words,
\begin{equation}
	\label{eq:qf-hl}
  Q_n[f;x] = H_{w^2}[L_n[f];x]
\end{equation}
where $L_n[f]$ is the Lagrange interpolation polynomial for the function $f$ with nodes $x_{j,n}$.
Let
\[
  E_n\left[f\right]_{w,\infty} := \inf_{P\in \mathcal{P}_n} \left \| \left( f - P \right) w \right \|_{L_{\infty}(\mathbb R)}
\]
be the error of best weighted polynomial approximation of a given $f$. 

We shall prove the following error bound for this numerical approximation scheme for the singular integral:

\begin{theorem}\label{Theorem 2.3}
  Let $w$ be admissible and let $\xi \in (0,1)$ be fixed.
  Consider a sequence $\mu_n:=(\mu_n)_{n=1}^{\infty}$ which converges to $0$ as $n \to \infty$ and satisfies
  $0 < \mu_n \le \min \{a_n \eta_n, |a_{-n}|\eta_{-n} \}$ for each $n\geq 1$. 
  Let $x \in \mathbb R$ and let $f\in \mathrm{Lip}(\gamma, w)$ for some $\gamma >0$. Then uniformly for $n$ large enough,
  \begin{equation*}
    |R_n[f;x]|
    \le C \left[ (1+  n^{-2}  w^{-1}(x)) \log n + \gamma_n(x) \right] E_{n-1}[f]_{w,\infty},
  \end{equation*}
  where
  \begin{eqnarray*}
    \gamma_n(x) & := & \mu_n^{-1} \delta_n^{5/4} \max \left \{ \frac{T(a_n)}{a_n}, \frac{T(a_{-n})}{|a_{-n}|} \right\}^{1/4} \log n\\
    && {} + \mu_n \times
    \begin{cases}
      A_n, & a_n ( 1 + C \eta_n) \le x \le 2a_n, \\
      B_n,  & a_{\xi n} \le x \le a_n ( 1 + C \eta_n), \\
      C_n,  &  a_{-\xi n} \le x \le a_{\xi n}, \\
      D_n,  & a_{-n} ( 1 + C \eta_{-n}) \le x \le a_{-\xi n}, \\
      E_n, & 2a_{-n} \le x \le a_{-n} ( 1 + C \eta_{-n} ), \\
      0,  & \textrm{otherwise,}
    \end{cases}
  \end{eqnarray*}
  and
  \begin{eqnarray*}
    A_n & := & n^{5/6} \delta_n^{1/3} a_n^{-5/6} T^{5/6}(a_n)
               \max \left \{ \frac{T(a_n)}{a_n}, \frac{T(a_{-n})}{|a_{-n}|} \right \}^{1/2},\\
    B_n & := & n \delta_n^{1/4} a_n^{-3/4} T^{3/4}(a_n)
               \max \left \{ \frac{T(a_n)}{a_n}, \frac{T(a_{-n})}{|a_{-n}|} \right \}^{1/2},\\
    C_n & := & n^{7/6} \delta_n^{-1/3}
               \max \left \{ \frac{T(a_n)}{a_n}, \frac{T(a_{-n})}{|a_{-n}|} \right \}^{2/3} \log n,\\
    D_n & := & n \delta_n^{1/4} |a_{-n}|^{-3/4} T^{\frac34}(a_{-n})
               \max \left \{ \frac{T(a_n)}{a_n}, \frac{T(a_{-n})}{|a_{-n}|} \right \}^{1/2},\\
    E_n & := & n^{\frac56} \delta_n^{1/3} |a_{-n}|^{-5/6} T^{\frac56}(a_{-n})
               \max \left \{ \frac{T(a_n)}{a_n}, \frac{T(a_{-n})}{|a_{-n}|} \right \}^{1/2}.
  \end{eqnarray*}
\end{theorem}

\subsection{The sequence of functions $\gamma_n(\cdot)$}

In this section, we look at the sequence of functions $\gamma_n(.)$ for the two examples $Q_{1,\alpha,\beta}$ and $Q_{2,\alpha,\beta,\ell,k}$ defined in Section 1, so that the reader may absorb Theorem \ref{Theorem 2.3}. We recall the definitions of $Q_{1,\alpha,\beta}$ and $Q_{2,\alpha,\beta}$ and some information re these.

Let $\beta \ge \alpha > 1$ and $\ell,k\geq 0$. Then:
\begin{equation*}
Q_{1,\alpha,\beta}(x):=
        \begin{cases}
	  x^{\alpha}, & x \in [0,\infty) \\
	  |x|^{\beta}, & x \in (-\infty,0).
        \end{cases} 
\end{equation*}

and 

\begin{equation*}
Q_{2,\alpha,\beta,\ell,k}(x):=
        \begin{cases}
	  \exp_l(x^{\alpha})-\exp_l(0), & x \in [0,\infty) \\
	  \exp_k(|x|^{\beta})-\exp_k(0), & x \in (-\infty,0).
        \end{cases}
\end{equation*}

Then straightforward calculations yield the following properties uniformly for $n$ large enough.
\medskip

\begin{itemize} 
\item $Q_{1,\alpha,\beta}:$
\begin{itemize}
\item[(a)] $T(a_n)=\alpha$.
\item[(b)] $T(a_{-n})=\beta$.
\item[(c)] $a_n\sim n^{\frac{1}{\alpha}}.$
\item[(d)] $|a_{-n}|\sim n^{\frac{1}{\alpha}.
\left(\frac{2\alpha-1}{2\beta-1}\right)}$.
\item[(e)] $\delta_n\sim a_n\sim n^{\alpha}$
\item[(f)] $\eta_n\sim n^{-2/3}$.
\item[(g)] $\eta_{-n}\sim \left[n^{1-\frac{1}{\alpha}.
\left(\frac{\beta-\alpha}{2\beta-1}\right)}\right]^{-2/3}$.
\end{itemize}
\end{itemize}

\begin{itemize} 
\item $Q_{2,\alpha,\beta,k,\ell}:$
\begin{itemize}
\item[(a)] $T(a_n)\sim \prod_{j=1}^{\ell}\log_{j}(n)$.
\item[(b)] $T(a_{-n})\sim \prod_{j=1}^{k}\log_{j}(n)$.
\item[(c)] $a_n = (\log_{\ell}(n))^{1/\alpha} (1+o(1))$.
\item[(d)] $a_{-n} = -(\log_{k}(n)^{1/\beta} (1+o(1)).$
\item[(e)] $\delta_n\sim a_n\sim (\log_{\ell}(n))^{1/\alpha} (1+o(1))$
\item[(f)] $\eta_n\sim \left(n\prod_{j=1}^{\ell}\log_{j}(n)
\right)^{-2/3}$.
\item[(g)] $\eta_{-n}\sim \left(n\prod_{j=1}^{k}\log_j(n)\left[\frac{\left(\log_k(n)\right)^{1/\beta}}{\left(\log_l(n)\right)^{1/\alpha}}\right]^{1/2}
\right)^{-2/3}.$
\end{itemize}
\end{itemize}

\begin{remark}
In the case when $Q$ is even and $T\sim 1$, that is $Q$ is of smooth polynomial growth with large argument (for example $Q_{1,\alpha,\beta}$), Theorem \ref{Theorem 2.3} is essentially Theorem 1.3 of \cite{DD1}.
\end{remark}

\setcounter{equation}{0}
\section{Partial Proof of Theorem \ref{Theorem 2.3} and some other Main Results}

In this section, we provide the necessary machinery for the partial proof of Theorem \ref{Theorem 2.3} and along the way state and 
prove several other main results. We need the following two lemmas taken from \cite{LL}. 
\begin{lemma}\label{Lemma 2.3}
 Let $w$ be admissible. Set for $n\geq 1$:
  \[ 
    h_n := \frac{n}{\sqrt{\delta_n}} \max \left \{ \frac{T(a_n)}{a_n}, \frac{T(a_{-n})}{|a_{-n}|} \right \}^{1/2} 
  \]
  and
  \[
    k_n := n^{1/6} \delta_n^{-1/3} \max \left \{ \frac{T(a_n)}{a_n}, \frac{T(a_{-n})}{|a_{-n}|} \right \}^{1/6}.
  \]
Then the following hold: 

 \begin{enumerate}
  \item[(a)] Let $0 < p \le \infty$. Then for $n \ge 1$ and $P \in \mathcal{P}_n$,
    \begin{equation*}
      \| P'w \|_{L_p(\mathbb R)} \le C h_n \| Pw \|_{L_p(\mathbb R)}.
    \end{equation*}
  \item[(b)] For $n\geq 1$,
    \begin{equation*}
      \sup_{x\in \mathbb R} |p_n(x) w(x)| \cdot |(x-a_{-n})(a_n-x)|^{1/4} \sim 1.
    \end{equation*}
  \item[(c)] For $n\geq 1$,
    \begin{equation*}
      \sup_{x\in \mathbb R} |p_n(x) w(x)| \sim k_n.
    \end{equation*}
  \item[(d)] For $n\geq 1$,
    \begin{equation*}
      \|p_n w\|_{L_p(\mathbb R)} \sim
      \begin{cases}
        \delta_n^{\frac1p-\frac12}, & 0 < p < 4, \\
        \delta_n^{-1/4}(\log (n+1))^{1/4}, & p = 4, \\
        \delta_n^{\frac1{3p}-\frac13} \max \left \{ \frac{T(a_n)}{a_n}, \frac{T(a_{-n})}{|a_{-n}|} \right \}^{\frac23(\frac14-\frac1p)}, & p > 4.
      \end{cases}
    \end{equation*}
  \item[(e)] For large enough $t>0$, there exists large enough $L>0$ such that uniformly in $t$, 
    \[
      \left| \delta_{t} \frac{T(a_{\pm t})}{a_{\pm t}} \right| \sim \left( \frac{t}{Q(a_{\pm t})} \right)^2 \le C t^{2(1-\frac{\log 2}{\log L})}
    \] 
and thus
 \[
      Q(a_{\pm t}) \ge C t^{\frac{\log 2}{\log L}}
      \quad \mbox{ and } \quad
      w(a_{\pm t}) = \exp(-Q(a_{\pm t})) \le \exp \! \left(-C t^{\frac{\log 2}{\log L}} \right).
    \]
  \item[(f)] The following hold:
\begin{itemize}
\item[(f1)] For large enough $t>0$, 
    \[
      \left| Q'(a_{\pm t}) \right| \sim t \sqrt{\frac{T(a_{\pm t})}{\delta_t |a_{\pm t}|}} \le C t^{2}. 
    \]
\item[(f2)] For fixed $L>0$, $j=0,1$ and uniformly for $t>0$, 
\[
Q^{(j)}(a_{Lt})\sim  Q^{(j)}(a_{t})
\]
and
\[T(a_{Lt})\sim T(a_{t})\].
\item[(f3)] For $t\neq 0$ and $\frac{1}{2}\leq \frac{t}{s}\leq 2$, 
\[
\left|1-\frac{a_s}{a_u}\right|\sim \frac{1}{T(a_t)}\left|1-\frac{s}{t}\right|.
\]
\item[(f4)] For fixed $L>1$ and uniformly for $t>0$, 
$a_{Lt}\sim a_t.$
\end{itemize}
\item[(g)] For $n\geq 1$,
    \begin{equation*} 
      |p_n' (x_{j,n}) w(x_{j,n})| \sim \varphi_n^{-1}(x_{j,n}) | (x_{j,n} - a_{-n}) (a_n - x_{j,n}) |^{-1/4}
    \end{equation*}
    holds uniformly for all $1\leq j\leq n$.
  \item[(h)] For $n\geq 2$,
    \begin{equation*} 
      x_{j,n} - x_{j+1,n} \sim \varphi_n(x_{j,n})
    \end{equation*}
    holds uniformly for all $2\leq j\leq n$.
  \item[(i)]
    For $x \in (x_{j+1, n}, x_{j, n})$, $n\geq 1$ and $1\leq j\leq n$, 
    \begin{eqnarray*}
      | p_n(x) w(x) | & \le & C \min\{ |x-x_{j,n}|, |x-x_{j+1,n}| \} \\
      && \qquad \times \varphi_n^{-1}(x_{j,n})|(x_{j,n}-a_{-n})(a_n-x_{j,n})|^{-1/4}.
    \end{eqnarray*}
  \item[(j)] Let  $0 < \alpha < 1$ and $n\geq 1$.  For $x \in (a_{-\alpha n},a_{\alpha n})$,
    \begin{equation*}
      | (x-a_{-n}) (a_n-x) |^{-1} \le C \frac1{\delta_n} \max \left \{ \frac{T(a_n)}{a_n}, \frac{T(a_{-n})}{|a_{-n}|} \right \}.
    \end{equation*}
  \end{enumerate}
\end{lemma}

\begin{lemma} \label{Lemma 2.4}
  Let $w$ be admissible, $0 < p < \infty$, $0 < \alpha < 1$, and $L >0$.
  Denote the $p$th power Christoffel functions by
  $\lambda_{n,p}(w,x):=\inf_{P\in  \mathcal{P}_n}\left(\frac{\|Pw\|_{L_{p}(\mathbb R)}}{P(x)}\right)^{p}$ 
  for $n\geq 1$ and $x\in \mathbb R$.
  Then uniformly for $n \ge 1$ and $x \in [a_{-n} (1 + L \eta_{-n}), a_n (1 + L \eta_n)]$,
  \[
    \lambda_{n,p}(w, x) \sim \varphi_n(x) w^p(x).
  \]
  Moreover, there exist $C$, $n_0 > 0$ such that uniformly for $n \ge n_0$ and $x \in \mathbb R$,
  \[
    \lambda_{n,p}(w, x) \ge C \varphi_n(x) w^p(x).
  \]
\end{lemma}

\subsection{Functions of  the second kind}

Let $w$ be admissible and let $p_n$ be the $n$th degree orthonormal polynomial for $w^2$. 
We define a sequence of functions of the second kind $q_n : \mathbb R \to \mathbb R$, $n\geq 1$ by 
(cf.\ \cite{CDLM,DD1,DD2,DD3,DD4})
\[
  q_n(x) := \cpv_{\mathbb R} w^2(t) \frac{p_n(t)}{t-x} dx, \qquad n\geq 1.
\]

The following holds:

\begin{theorem}\label{Theorem 2.4}
  Let $w$ be admissible and $0 < \xi < 1$. Then,
  \begin{equation}\label{q'}
    |q_n'(x)|
      \le C \! \times \! 
        \begin{cases}
          n^{7/6} \delta_n^{-5/6} \max \left \{ \frac{T(a_n)}{a_n}, \frac{T(a_{-n})}{|a_{-n}|} \right \}^{2/3} \log h_n,
                & a_{-\xi n} \le x \le a_{\xi n}, \\
          n a_n^{-1} T(a_n) \max \left \{ \frac{T(a_n)}{a_n}, \frac{T(a_{-n})}{|a_{-n}|} \right \}^{1/2},
                & x \ge a_{\xi n}, \\
          n |a_{-n}|^{-1} T(a_{-n}) \max \left \{ \frac{T(a_n)}{a_n}, \frac{T(a_{-n})}{|a_{-n}|} \right \}^{1/2},
                & x \le a_{-\xi n}.
        \end{cases}
  \end{equation}
\end{theorem}

\begin{proof}
  Let for $x\in \mathbb R$ and $n\geq 1$,
  \[
    \rho_n(x) := w^2(x) p_n(x).
  \]
  Then
  \[
    q_n'(x) = \cpv_{\mathbb R} \rho_n'(t) \frac{1}{t-x} dt.
  \]
  Introducing a positive sequence $\varepsilon_n$ that we shall define precisely later, we write for $x\in \mathbb R$ and $n\geq 1$
  \[
    q_n'(x) = A_1+A_2:=A_1(x) + A_2(x)
  \]
  with
  \[
    A_1 = \int_{|t-x| \ge \varepsilon_n} \rho_n'(t) \frac{1}{t-x} dt
  \]
  and
  \[
    A_2 = \cpv_{|t-x|< \varepsilon_n} \rho_n'(t) \frac{1}{t-x} dt
        = \cpv_{x- \varepsilon_n}^{x + \varepsilon_n} \frac{\rho_n'(t) - \rho_n'(x)}{t-x} dt.
  \]

  Let us collect some auxiliary results: Since for $x\in \mathbb R$
  \[
    \rho_n'(x) = p_n'(x)w^2(x) - 2 Q' p_n(x)w^2(x)
  \]
  and
  \[
    \rho_n''(x) = p_n''(x)w^2(x)- 4Q'(x)p_n'(x) w^2(x) + \left(- 2Q''(x) + 4 {Q'}^2(x) \right) p_n(x) w^2(x),
  \]
  we have for $x \in \mathbb R$, in view of Lemma \ref{Lemma 2.3}(a) and the definition of $h_n$ given in the preamble of Lemma \ref{Lemma 2.3},
  \begin{eqnarray*}
    | \rho_n'(x) w^{-1/2}(x) | 
      & \le & | p_n'(x) w(x) w^{1/2}(x) | + 2 | Q'(x) w^{1/2}(x) |\cdot | p_n(x) w(x) | \\
      & \le & C \left( h_n w^{1/2}(x) + |Q'(x) w^{1/2}(x) | \right) \| w p_n \|_{L_{\infty}(\mathbb R)}
  \end{eqnarray*}
  Continuing,
  \[
    |\rho_n'(x)|
      \le C \left( h_n w(x) + |Q'(x) w(x)| \right) \| w p_n \|_{L_{\infty}(\mathbb R)}
      \le C h_n\|w p_n\|_{L_{\infty}(\mathbb R)}.
  \]
  Using (d) and (e) of Definition \ref{Definition 1.1} and the definition of $h_n$, we have for $x\in \mathbb R$,
  \begin{eqnarray*}
    | Q''(x) w(x) | 
      & \le & C \frac{(Q'(x))^2}{|Q(x)|} w(x) \\
      &  =  & C |Q'(x)| \frac{T(x)}{x} w(x) \\
      & \le & C |Q'(x)| w(x) \max \left \{ \frac{T(a_n)}{a_n}, \frac{T(a_{-n})}{|a_{-n}|} \right \} \\
      & \le & C |Q'(x)| w(x) h_n^2 \frac{\delta_n}{n^2} \\
      & \le & C |Q'(x)| w(x) h_n
  \end{eqnarray*}
  and hence we see that
  \begin{eqnarray*}
    \lefteqn{|\rho_n''(x)|} \\
      & \le & C \left( h_n^2 w(x) + h_n | Q'(x) w(x) | + | Q''(x) w(x) | + | Q'(x) w^{1/2}(x) |^2 \right) \| w p_n \|_{L_{\infty}(\mathbb R)} \\
      & \le & C \left( h_n^2 w(x) + h_n | Q'(x) w(x) | + | Q'(x) w^{1/2}(x) |^2 \right) \| w p_n \|_{L_{\infty}(\mathbb R)}.
  \end{eqnarray*}
  Finally, in view of Lemma \ref{Lemma 2.3}(e), Lemma \ref{Lemma 2.3}(f), and the relations
  $\lim_{t \to \infty} a_t = \infty$ and $\lim_{t \to \infty} a_{-t} = -\infty$  (which follow from their definitions) 
  we see that
  \begin{equation}\label{Q'w}
    |Q'(x) w^{1/2}(x)| \le C, \quad x \in \mathbb R.
  \end{equation}

  Now we are in a position to deal with the case $a_{-\xi n} \le x \le a_{\xi n}$. We
 apply H\"older's inequality and derive
  \[
    |A_1| 
      = \left| \int_{|t-x|\ge \varepsilon_n} \rho_n'(t) \frac{1}{t-x} dt \right|
      \le \left\| \rho_n'(t) w^{-1/2} \right \|_{L_{\infty}(\mathbb R)} \left \| (\cdot-x)^{-1} w^{1/2} \right \|_{L_1(S_n)}
  \]
  where $S_n = \{t \in \mathbb R : |t-x| \ge \varepsilon_n \}$.
  An explicit calculation gives
  \[
    \left\|(\cdot-x)^{-1}w^{1/2}\right\|_{L_1(S_n)} \sim |\log \varepsilon_n|
  \]
  uniformly in $n$. Then
  \[
    |A_1| \le C h_n \| w p_n \|_{L_{\infty}(\mathbb R)} |\log \varepsilon_n|.
  \]
  For $A_2$, we observe that
  \[
    |A_2| \le 2 \varepsilon_n \| \rho_n'' \|_{L_{\infty}(\mathbb R)} \le C \varepsilon_n h_n^2 \| w p_n \|_{L_{\infty}(\mathbb R)}.
  \]
  Bearing in mind that by definition, $h_n \to \infty$ as $n \to \infty$, we see
  that $h_n^{-1}$ can be made arbitrarily small for sufficiently large $n$. Then
  \[
    |q_n'(x)| \le C h_n \| w p_n \|_{L_{\infty}(\mathbb R)} \log h_n \sim C h_n k_n \log h_n
  \]
  because of Lemma \ref{Lemma 2.3}(c).

  In view of the definitions of $h_n$ and $k_n$, this completes the proof in the first case.

  Next, we consider the case $x \ge a_{\xi n}$. 

  
 We define for $n\geq 1$, the sequence $\varepsilon_n := a_{\xi n} - a_{\xi n/2}$. The behavior of $\varepsilon_n$ uniformly for large enough $n$ is determined by Lemma \ref{Lemma 2.3} (f3) which we recall says the following:
 \medskip
 
For $t\neq 0$ and $\frac{1}{2}\leq \frac{t}{s}\leq 2$, 
\[
\left|1-\frac{a_s}{a_u}\right|\sim \frac{1}{T(a_t)}\left|1-\frac{s}{t}\right|.
\]

In particular, for $Q(1,\alpha,\beta)$, the sequence $\varepsilon_n$ grows without bound uniformly for large $n$ and for $Q(2,\alpha,\beta,\ell,k)$, $\varepsilon_n$ tends to $0$ uniformly for large $n$.
\medskip 
We write
  \begin{equation}
    q_n'(x) = \left( \int_{|t-x|\ge \varepsilon_n} + \int_{n^{-10} <|t-x|< \varepsilon_n} + \int_{|t-x|\le n^{-10}} \right) \frac{\rho'(t)}{t-x} dt  
            = A_3 + A_4 + A_5.
  \end{equation}
  Note that this decomposition is possible since we have by definition $a_{\xi n} - a_{\xi n/2} \sim a_{\xi n} / T(a_{\xi n}) > n^{-10}$ 
  for sufficiently large $n$. 
  For $A_5$, we argue in a similar way as for $A_2$ above and find
  \begin{eqnarray*}
    |A_5| &  =  & \int_{x-n^{-10}}^{x+n^{-10}} \frac{\rho_n'(t) - \rho_n'(x)}{t-x} dt
            \le   C n^{-10} \| \rho_n'' \|_{L_{\infty}(\mathbb R)} \\
          & \le & C n^{-10} h_n^2 \| w p_n \|_{L_{\infty}(\mathbb R)} \\
          & \le & C n^{-10} h_n^2 k_n \le O(n^{-2})
  \end{eqnarray*}
  where the last inequality follows from the fact that, in view of Lemma \ref{Lemma 2.3}(f),
  \[
    \max \left \{ \frac{T(a_n)}{a_n}, \frac{T(a_{-n})}{|a_{-n}|} \right \} \le C n^2 \delta_n
  \]
  and hence
  \[
    h_n^2 k_n \le C n^{9/2} \delta_n^{-1/6} \le C n^{9/2} 
  \]
  since $\delta_n$ is an increasing sequence of positive numbers.
  For $A_4$,
  \begin{eqnarray*}
    |A_4| 
      &  =  & \int_{n^{-10} <|t-x|< \varepsilon_n} \frac{1}{|t-x|} \| \rho_n' \|_{L_{\infty}(n^{-10} <|t-x|< \varepsilon_n)} dt \\
      & \le & C \| \rho_n' \|_{L_{\infty}(n^{-10} < |t-x| < \varepsilon_n)} |\log \varepsilon_n| \\
      & \le & C h_n k_n w \left( a_{\xi n/2} \right) |\log \varepsilon_n|
        \le   C h_n k_n o(\exp(-n^{\alpha_2})) \textrm{ for some } \alpha_2 > 0.
  \end{eqnarray*}
  Here, we used the following properties: For $a_{\xi n/2} < t < 2 a_{\xi n} - a_{\xi n/2}$
  we have
  \begin{eqnarray*}
    \lefteqn{\| \rho_n' \|_{L_{\infty}(n^{-10} < |t-x| < \varepsilon_n)}} \\
      & \le & C \| \rho_n' \|_{L_{\infty}(x \ge  a_{\xi n/2})} \\
      & \le & C \sup_{x \in \mathbb R, \, x \ge a_{\xi n / 2}} \big| h_n w(x) + | Q'(x) w(x) | \, \big| \cdot \|w p_n \|_{L_{\infty}(\mathbb R)} \\
      & \le & C \sup_{x \in \mathbb R, \, x \ge a_{\xi n / 2}} \!\!\!\!\! w^{1/2}(x) 
                   \cdot \sup_{x \in \mathbb R, \, x \ge a_{\xi n / 2}}  \!\!\!\!\! \big| h_n w^{1/2}(x) + | Q'(x) w^{1/2}(x) | \, \big| \cdot \|w p_n \|_{L_{\infty}(\mathbb R)} \\
      & \le & C h_n w^{1/2} (a_{\xi n/ 2}) \| w p_n \|_{L_{\infty}(\mathbb R)} 
        \quad (\mbox{because } |Q' w^{1/2}| \mbox{ and } w^{1/2} \mbox{ are bounded} \\
      &     & \mbox{ \hspace{5cm} and } h_n \to \infty) \\
      & \le & C h_n k_n w^{1/2} (a_{\xi n/2}) \\
      & \le & C h_n k_n o(\exp(-n^{\alpha_2})) \mbox{ for some }\alpha_2 >0,
  \end{eqnarray*}
  because $w^{1/2} (a_{\xi n/2}) = O(\exp(-n^{\alpha_1}))$ for some $\alpha_1 >0$.
  
  Finally, for $|A_3|$, since $\varepsilon_n \sim a_n/T(a_n)$
  and using
  \[
    \left| (p_n w^2)'(x) \right| \le C \left( \left| p_n'(x) w(x) \right| + \left| p_n(x) w(x) \right| \right)
    \quad \mbox{ because } \| Q' w \|_{L_\infty(\mathbb R)} < \infty,
  \]
  we see
  \begin{eqnarray*}
    \| \rho_n' \|_{L_1(\mathbb R)} = \| (p_nw^2)' \|_{L_1(\mathbb R)} \le C h_n \| p_n w \|_{L_1(\mathbb R)} \sim h_n \delta_n^{1/2},
  \end{eqnarray*}
  and so we have
  \begin{eqnarray*}
    |A_3| & \le & \frac{1}{\varepsilon_n} \int_{\mathbb R} |\rho_n'(t)| dt
      \le C \frac{T(a_n)}{a_{n}} \| \rho_n' \|_{L_{1}(\mathbb R)}
      \le C \frac{T(a_n)}{a_{n}} h_n \delta_n^{1/2}.
  \end{eqnarray*}
  Then for $x \ge a_{\xi n}$
  \[
    |q_n'(x)| \le C \frac{T(a_n)}{a_{n}} h_n \delta_n^{1/2}.
  \]

  Finally, for $x \le a_{-\xi n}$ we proceed in almost the same way as for $x \ge a_{\xi n}$; the only difference is that we now obtain
  \[
    |q_n'(x)| \le C \frac{T(a_{-n})}{|a_{-n}|} h_n \delta_n^{1/2}.
  \]
 
  Therefore, we can summarize the results as follows. For $n$ large enough,
  \begin{eqnarray*}
    |q_n'(x)| 
      & \le & C \times
                 \begin{cases}
                   C h_n \| w p_n \|_{L_{\infty}(\mathbb R)} \log n,  & a_{-\xi n} \le x \le a_{\xi n}\\
                   \frac{T(a_n)}{a_{n}} h_n \delta_n^{1/2}, & x \ge a_{\xi n}\\
                   \frac{T(a_{-n})}{|a_{-n}|} h_n \delta_n^{1/2}, & x \le a_{-\xi n}.
                 \end{cases} \\
      & \sim & \begin{cases}
                   n^{7/6} \delta_n^{-5/6} \max \left \{ \frac{T(a_n)}{a_n}, \frac{T(a_{-n})}{|a_{-n}|} \right \}^{2/3} \log n,
                           & a_{-\xi n} \le x \le a_{\xi n} \\
                   n a_n^{-1} T(a_n) \max \left \{ \frac{T(a_n)}{a_n}, \frac{T(a_{-n})}{|a_{-n}|} \right\}^{1/2},
                           & x \ge a_{\xi n} \\
                   n |a_{-n}|^{-1} T(a_{-n}) \max \left \{ \frac{T(a_n)}{a_n}, \frac{T(a_{-n})}{|a_{-n}|} \right\}^{1/2},
                           & x \le a_{-\xi n}.
                 \end{cases}
  \end{eqnarray*}
\end{proof}

We also need the following result, cf.\ \cite[Lemma 3.1]{DD1}.

\begin{lemma}\label{Lemma 2.5}
  For the weights $\zeta_{jn}(\cdot)$ of the quadrature formula $Q_n$ defined in (\ref{w_jn}), we have for $x\in \mathbb R$ and $1\leq j\leq n$
  \[
    \zeta_{jn}(x) =
    \begin{cases}
      \displaystyle \frac{q_n(x_{j,n}) - q_n(x)}{(x_{j,n} - x) p_n'(x_{j,n})}, & \textrm{ if } x \neq x_{j,n},\\
      \displaystyle \frac{q_n'(x_{j,n})}{p_n'(x_{j,n})}, & \textrm{ if } x = x_{j,n}.
    \end{cases}
  \]
\end{lemma}

\begin{proof}
This result has been shown in \cite[Lemma 3.1]{DD1} for a narrower class of weight functions than the class 
under consideration here. However, the proof given there only exploits properties that are satisfied in the 
present situation too, and so it can be carried over directly.
\end{proof}

\begin{theorem}\label{Theorem 2.6}
Let $w$ be admissible. Let $x \in \mathbb R$, $n\geq 2$ and let $f\in \mathrm{Lip}(\gamma, w)$ for some $\gamma >0$.
Let $P_{n-1}^* \in \mathcal P_{n-1}$ satisfy
\[
\|w(f-P_{n-1}^*)\|_{L_{\infty}(\mathbb R)}
=\inf_{P\in \mathcal{P}_{n-1}}\|w(f-P)\|_{L_{\infty}(\mathbb R)}
=E_{n-1}[f]_{w,\infty}.
\]
Then, for $x\in \mathbb{R}$
\[
	|H_{w^2}[f-P_{n-1}^*;x] | \le C(1+  n^{-2}  w^{-1}(x))E_{n-1}[f]_{w,\infty} \log n,
\]
where $C$ is a constant depending on $w$, but independent of $f$, $n$, and $x$.
\end{theorem}

\begin{proof}
We estimate:

\[
|H_{w^2}[f-P_{n-1}^*;x]|.
\]
Let $x\in \mathbb R$ and $0<\varepsilon <\infty$. Then we have
\begin{eqnarray*}
\lefteqn{\left|\int_{|t-x|\ge \varepsilon} w^2(t)\frac{f(t)-P_{n-1}^*(t)}{t-x}dt\right|} \\
&\le&  CE_{n-1}[f]_{w,\infty}
\left(\|w\|_{L_{\infty}(\mathbb R)}\log \varepsilon^{-1} \right)
\end{eqnarray*}
and
\begin{eqnarray*}
\lefteqn{ \cpv_{x-\varepsilon}^{x+\varepsilon}w^2(t)\frac{f(t)-P_{n-1}^*(t)}{t-x}dt } \\
&=&\cpv_{x-\varepsilon}^{x+\varepsilon}w^2(t)\frac{(f(t)-P_{n-1}^*(t))-(f(x)-P_{n-1}^*(x))}{t-x}dt \\
&&\quad +w(x)(f(x)-P_{n-1}^*(x))\cpv_{x-\varepsilon}^{x+\varepsilon} \frac{w^2(t)}{w(x)}\frac{dt}{t-x}.
\end{eqnarray*}
Now, 
we see that because of the weighted Lipschitz condition on $f$,
\begin{equation*}
w(t)\frac{\left|(f(t)-P_{n-1}^*(t))-(f(x)-P_{n-1}^*(x))\right|}{|t-x|^{\alpha^*}}
\le C\varepsilon^{-\gamma^*} E_{n-1}[f]_{w,\infty}
\end{equation*}
for some $\gamma^* \in (0,\gamma/2)$. Thus,
\begin{eqnarray}\label{(1)} \nonumber
\lefteqn{\left|\cpv_{x-\varepsilon}^{x+\varepsilon}w^2(t)
\frac{(f(t)-P_{n-1}^*(t))-(f(x)-P_{n-1}^*(x))}{t-x}dt\right|} \\ \nonumber
&\le& \|w\|_{L_{\infty}(\mathbb R)} \\ \nonumber
&& \quad \times \cpv_{x-\varepsilon}^{x+\varepsilon}w(t)
\frac{\left|(f(t)-P_{n-1}^*(t))-(f(x)-P_{n-1}^*(x))\right|}{|t-x|^{\gamma^*}|t-x|^{1-\gamma^*}}dt\\ \nonumber
&\le& C\varepsilon^{-\gamma^*}
\|w\|_{L_{\infty}(\mathbb R)} E_{n-1}[f]_{w,\infty}
\int_0^{\varepsilon}y^{\gamma^*-1}dy\\ \nonumber
&\le& C\varepsilon^{-\gamma^*}
\|w\|_{L_{\infty}(\mathbb R)} E_{n-1}[f]_{w,\infty} \varepsilon^{\gamma^*}\\
&\le& C\|w\|_{L_{\infty}(\mathbb R)} E_{n-1}[f]_{w,\infty}.
\end{eqnarray}
Now, we estimate
\begin{equation*}
\left|\cpv_{x-\varepsilon}^{x+\varepsilon} \frac{w^2(t)}{w(x)}\frac{dt}{t-x}\right|.
\end{equation*}
Using the mean value theorem, there exists $t_x$  between $t$ and $x$ such that
\begin{eqnarray}\label{(2)} \nonumber
\left|\cpv_{x-\varepsilon}^{x+\varepsilon} \frac{w^2(t)}{w(x)}\frac{dt}{t-x}\right|
&=&w^{-1}(x)\left|\cpv_{x-\varepsilon}^{x+\varepsilon} \frac{w^2(t)-w^2(x)}{t-x}dt\right|\\ \nonumber
&=&w^{-1}(x)\left|\cpv_{x-\varepsilon}^{x+\varepsilon} -2Q'(t_x)w^2(t_x)dt\right|\\
&\le& Cw^{-1}(x)\|Q'w^2\|_{L_{\infty}(\mathbb R)}\varepsilon
\le Cw^{-1}(x)\varepsilon.
\end{eqnarray}
Therefore, we have by (\ref{(1)}) and (\ref{(2)})
\begin{eqnarray*}
\lefteqn{ \left|\cpv_{x-\varepsilon}^{x+\varepsilon}w^2(t)\frac{f(t)-P_{n-1}^*(t)}{t-x}dt\right| } \\
&\le& C_1\|w\|_{L_{\infty}(\mathbb R)} E_{n-1}[f]_{w,\infty}
+ C_2w^{-1}(x)\varepsilon E_{n-1}[f]_{w,\infty} \\
&\le& C\left(\|w\|_{L_{\infty}(\mathbb R)} + \varepsilon  w^{-1}(x)\right)E_{n-1}[f]_{w,\infty}.
\end{eqnarray*}
Consequently, we have taking $\varepsilon=n^{-2}$,
\begin{eqnarray*}
\lefteqn{|H_{w^2}[f-P_{n-1}^*;x]|}\\
&=& \left|\cpv_{\mathbb R} w^2(t)\frac{f(t)-P_{n-1}^*(t)}{t-x}dt \right| \\
&\le&
\left|\int_{|t-x|\ge \varepsilon} w^2(t)\frac{f(t)-P_{n-1}^*(t)}{t-x}dt\right|
 + \left|\cpv_{x-\varepsilon}^{x+\varepsilon}w^2(t)\frac{f(t)-P_{n-1}^*(t)}{t-x}dt\right| \\
 &\le& CE_{n-1}[f]_{w,\infty}
\left(\|w\|_{L_{\infty}(\mathbb R)}\ln \varepsilon^{-1} \right)  \\
&& + C\left(\|w\|_{L_{\infty}(\mathbb R)} + \varepsilon  w^{-1}(x)\right) E_{n-1}[f]_{w,\infty} \\
&\le&C (1+ \varepsilon  w^{-1}(x))E_{n-1}[f]_{w,\infty} \ln \varepsilon^{-1} \\
&\le&C (1+  n^{-2}  w^{-1}(x))E_{n-1}[f]_{w,\infty} \ln n.
\end{eqnarray*}
\end{proof}

\begin{theorem}\label{Theorem 1.3_1}
Let $w$ be admissible and
assume that $0<\mu_n \! \le \! \min\{a_n \eta_n, |a_{-n}|\eta_{-n}\}$.
Let $x \in \mathbb R$, $n$ large enough and let $f\in \mathrm{Lip}(\gamma, w)$ for some $\gamma >0$. Then
\begin{equation*}
|R_n[f;x]|
\le C\left\{ (1+  n^{-2}  w^{-1}(x))\log n + \gamma_n(x)\right\}E_{n-1}[f]_{w,\infty},
\end{equation*}
where
\begin{equation*}
\gamma_n(x) :=  \|q_n\|_{\infty}\mu_n^{-1} \delta_n^{3/2}
+ \mu_n
\begin{cases}
A_n, & a_n(1+C\eta_n) \le x \le 2a_n, \\
B_n,  & a_{\xi n} \le x \le a_n(1+C\eta_n), \\
C_n,  &  a_{-\xi n} \le x \le a_{\xi n}, \\
D_n,  & a_{-n}(1+C\eta_{-n}) \le x \le a_{-\xi n}, \\
E_n, & 2a_{-n} \le x \le a_{-n}(1+C\eta_{-n}), \\
0, & \textrm{otherwise}.
\end{cases}
\end{equation*}
\end{theorem}

\begin{proof}
The proof of Theorem \ref{Theorem 1.3_1} is based on the fact that our quadrature formula $Q_n$ is of interpolatory type,
i.e., it is exact for all polynomials of degree $\le n-1$. Thus,
\begin{eqnarray*}
|R_n[f;x]|&=& |R_n[f-P_{n-1}^*;x]| \\
&\le& |H_{w^2}[f-P_{n-1}^*;x]|
+ \sum_{j=1}^n \frac{|\zeta_{jn}(x)|}{w(x_{j,n})}
[w(x_{j,n})|f(x_{j,n})-P_{n-1}^*(x_{j,n})|],
\end{eqnarray*}
where $P_{n-1}^*$ is the polynomial of the best uniform approximation for $f$ from $\mathcal{P}_{n-1}$
with respect to the weight function $w$.
Hence, we now have to prove that
\[
\sum_{j=1}^n \frac{|\zeta_{jn}(x)|}{w(x_{j,n})} \le C \gamma_n(x).
\]
Let $x >0$.
Assume $0<\mu_n \le C \min\{a_n \eta_n, |a_{-n}|\eta_{-n}\}$. Then
\[
\sum_{j=1}^n \frac{|\zeta_{jn}(x)|}{w(x_{j,n})}
=\sum_{\substack{j=1\\ |x-x_{j,n}| > \mu_n}}^n \frac{|\zeta_{jn}(x)|}{w(x_{j,n})}
+ \sum_{\substack{j=1\\ |x-x_{j,n}| \le \mu_n}}^n \frac{|\zeta_{jn}(x)|}{w(x_{j,n})}.
\]
For the first part, we have from Lemma \ref{Lemma 2.3}(g), Lemma \ref{Lemma 2.3}(h) and Lemma \ref{Lemma 2.5}
\begin{eqnarray*}
\lefteqn{\sum_{\substack{j=1\\ |x-x_{j,n}| > \mu_n}}^n \frac{|\zeta_{jn}(x)|}{w(x_{j,n})}} \\
&\le& C \sum_{\substack{j=1\\ |x-x_{j,n}| > \mu_n}}^n \frac{1}{w(x_{j,n})}
\left|\frac{q_n(x_{j,n})-q_n(x)}{(x_{j,n}-x)p_n'(x_{j,n})}\right| \\
&\le& C \|q_n\|_{\infty}\mu_n^{-1}
\sum_{j=1}^n (x_{j,n}-x_{j+1,n})|(x_{j,n}-a_{-n})(a_n-x_{j,n})|^{1/4}\\
&\le& C \|q_n\|_{\infty}\mu_n^{-1}
\int_{[a_{-n}(1-C\eta_{-n}),a_n(1-C\eta_{n})] } |(t-a_{-n})(a_n-t)|^{1/4} dt \\
&\le& C \|q_n\|_{\infty}\mu_n^{-1}\delta_n^{3/2}
\int_{[-1,1] } (1-u^2)^{1/4} dt \\
&\le& C \|q_n\|_{\infty}\mu_n^{-1} \delta_n^{3/2}.
\end{eqnarray*}
Now, we consider the second part:
\[
\sum_{\substack{j=1\\ |x-x_{j,n}| \le \mu_n}}^n \frac{|\zeta_{jn}(x)|}{w(x_{j,n})}.
\]
\noindent
{\bf Case 1.} $x \ge 2a_n$ :
We know that all the zeros of $p_n$ are in the interval
\[
[a_{-n}(1-C\eta_{-n}), a_n(1-C\eta_n)].
\]
Then
since
\[
\frac{\eta_n^{-1}}{T(a_n)}=\left(n^2\frac{a_n}{\delta_nT(a_n)}\right)^{1/3}=
O(n^{\alpha_1}), \quad \textrm{ for some } \alpha_1 >0,
\]
we have for some $\alpha_2, \alpha_3 >0$
\begin{eqnarray*}
|x-x_{j,n}| &\ge& a_n 
\ge \min \left\{ a_n, \frac{a_{2n}-a_n}{2}\right\} \\
&\ge& C \frac{a_n}{T(a_n)} \ge O(n^{\alpha_2})a_n\eta_n
\ge O(n^{\alpha_3})\mu_n.
\end{eqnarray*}
Therefore, we have shown that the second sum is empty. \\
For the other cases, by the mean value theorem,
there exists $\xi_{j,n}$ between $x$ and $x_{j,n}$ such that
\begin{eqnarray*}
\sum_{\substack{j=1\\ |x-x_{j,n}| < \mu_n}}^n \frac{|\zeta_{jn}(x)|}{w(x_{j,n})}
&\le& C \sum_{\substack{j=1\\ |x-x_{j,n}| < \mu_n}}^n \frac{1}{w(x_{j,n})}
\left|\frac{q_n(x_{j,n})-q_n(x)}{(x_{j,n}-x)p_n'(x_{j,n})}\right| \\
&\le& C \sum_{\substack{j=1\\ |x-x_{j,n}| < \mu_n}}^n
\left|\frac{q_n'(\xi_{j,n})}{p_n'(x_{j,n})w(x_{j,n})}\right|.
\end{eqnarray*}
\noindent
{\bf Case 2.} $a_n(1-C\eta_n) \le x \le 2a_n$ :
Since $\mu_n \le C a_n\eta_n$, we have
\begin{eqnarray*}
|x-x_{j,n}| < \mu_n
&\Rightarrow&   x-\mu_n <  x_{j,n} \le x+\mu_n \\
&\Rightarrow&   a_n(1-C\eta_n) <  x_{j,n} \le x+\mu_n.
\end{eqnarray*}
Then we have
\[
|(x_{j,n}-a_{-n})(a_n-x_{j,n})| \le \delta_n a_n\eta_n,
\]
and from (\ref{q'})
\[
|q_n'(\xi_{j,n})| \le C na_n^{-1}T(a_n) \max\left\{\frac{T(a_n)}{a_n}, \frac{T(a_{-n})}{|a_{-n}|} \right\}^{\frac12}.
\]
Using Lemma \ref{Lemma 2.3}(g) and Lemma \ref{Lemma 2.3}(h), we have
\begin{eqnarray*}
\lefteqn{ \sum_{\substack{j=1\\ |x-x_{j,n}| < \mu_n}}^n \left|\frac{q_n'(\xi_{j,n})}{p_n'(x_{j,n})w(x_{j,n})}\right| }\\
&\le& C n^{\frac56}\delta_n^{\frac1{3}}a_n^{-\frac{5}{6}}T^{\frac56}(a_n)
\max\left\{\frac{T(a_n)}{a_n}, \frac{T(a_{-n})}{|a_{-n}|} \right\}^{\frac12}
\sum_{\substack{j=1\\ |x-x_{j,n}| < \mu_n}}^n(x_{j,n}-x_{j+1,n}) \\
&\le& C\mu_n A_n.
\end{eqnarray*}
{\bf Case 3.} $ a_{\xi n} \le x \le a_n(1-C\eta_n)$ : Similarly to Case 2, since  $\mu_n \le C a_n\eta_n$,
there exists $0 < \eta_1 < 1$ such that
\begin{eqnarray*}
|x-x_{j,n}| < \mu_n
&\Rightarrow&   x-\mu_n <  x_{j,n} \le x+\mu_n \\
&\Rightarrow&   a_{\eta_1 n} <  x_{j,n} \le x+\mu_n  \le a_n(1+C\eta_n).
\end{eqnarray*}
Then we have
\[
|(x_{j,n}-a_{-n})(a_n-x_{j,n})| \le \delta_n(a_n-a_{\eta_1 n}) \le C \delta_n\frac{a_n}{T(a_n)},
\]
and from (\ref{q'})
\[
|q_n'(\xi_{j,n})| \le C na_n^{-1}T(a_n) \max\left\{\frac{T(a_n)}{a_n}, \frac{T(a_{-n})}{|a_{-n}|} \right\}^{\frac12}.
\]
Using Lemma \ref{Lemma 2.3}(g) and Lemma \ref{Lemma 2.3}(h), we have
\begin{eqnarray*}
\lefteqn{ \sum_{\substack{j=1\\ |x-x_{j,n}| < \mu_n}}^n \left|\frac{q_n'(\xi_{j,n})}{p_n'(x_{j,n})w(x_{j,n})}\right| }\\
&\le& Cn\delta_n^{\frac1{4}}a_n^{-\frac{3}{4}}T^{\frac34}(a_n)
\max\left\{\frac{T(a_n)}{a_n}, \frac{T(a_{-n})}{|a_{-n}|} \right\}^{\frac12}
\sum_{\substack{j=1\\ |x-x_{j,n}| < \mu_n}}^n(x_{j,n}-x_{j+1,n}) \\
&\le& C\mu_n B_n.
\end{eqnarray*}
{\bf Case 4.} $0 \le x \le a_{\xi n} $ : Since $\mu_n \le |a_{-n}|\eta_{-n}$ and $\mu_n \le a_{n}\eta_{n}$,
there exist constants $L >0$ and $0 < \eta_1, \eta_2 < 1$ independent of $n$ with
\begin{eqnarray*}
|x-x_{j,n}| < \mu_n
&\Rightarrow&   x-\mu_n <  x_{j,n} \le x+\mu_n \\
&\Rightarrow&   a_{-n}\eta_{-n} <  x_{j,n} \le x+\mu_n  \le a_{\xi n}(1+L\eta_n) \le  a_{\eta_2 n}\\
&\Rightarrow&   a_{-\eta_1 n} <  x_{j,n} \le x+\mu_n  \le   a_{\eta_2 n}.
\end{eqnarray*}
Then we have
\[
|(x_{j,n}-a_{-n})(a_n-x_{j,n})| \le \delta_n(a_n-a_{\eta_1 n}) \le \delta_n^2,
\]
and from (\ref{q'})
\[
|q_n'(\xi_{j,n})| \le C n^{\frac76}\delta_n^{-\frac5{6}}
\max\left\{\frac{T(a_n)}{a_n}, \frac{T(a_{-n})}{|a_{-n}|} \right\}^{\frac23}\log n.
\]
Using Lemma \ref{Lemma 2.3}(g) and Lemma \ref{Lemma 2.3}(h), we have
\begin{eqnarray*}
\lefteqn{\sum_{\substack{j=1\\ |x-x_{j,n}| < \mu_n}}^n \left|\frac{q_n'(\xi_{j,n})}{p_n'(x_{j,n})w(x_{j,n})}\right| } \\
&\le& Cn^{\frac76}\delta_n^{-\frac1{3}}
\max\left\{\frac{T(a_n)}{a_n}, \frac{T(a_{-n})}{|a_{-n}|} \right\}^{\frac23}\log n
\sum_{\substack{j=1\\ |x-x_{j,n}| < \mu_n}}^n(x_{j,n}-x_{j+1,n}) \\
&\le& C\mu_n C_n.
\end{eqnarray*}
Thus, we have for $x\geq 0$
\begin{eqnarray*}
\sum_{j=1}^n \frac{|\zeta_{jn}(x)|}{w(x_{j,n})}  &\le& C_1\|q_n\|_{\infty}\mu_n^{-1} \delta_n^{3/2}\\
&& {} + C_2 \mu_n \times
\begin{cases}
A_n, & a_n(1+C\eta_n) \le x \le \min\{\frac{d+a_n}{2}, 2a_n\} \\
B_n,  & a_{\xi n} \le x \le a_n(1+C\eta_n) \\
C_n,  &  0 \le x \le a_{\xi n}, \\
0,  & \textrm{otherwise }
\end{cases} \\
&\le&C\gamma_n(x)
\end{eqnarray*}
Similarly, we have for $x \le 0$
\begin{eqnarray*}
\sum_{j=1}^n \frac{|\zeta_{jn}(x)|}{w(x_{j,n})}  &\le& C\|q_n\|_{\infty}\mu_n^{-1} \delta_n^{3/2}\\
&& {} +C \mu_n \times
\begin{cases}
E_n,
& \max\{\frac{c+a_{-n}}{2}, 2a_{-n}\} \le x \le a_{-n}(1+C\eta_{-n}) \\
D_n,  & a_{-n}(1+C\eta_{-n}) \le x \le a_{-\xi n} \\
C_n,  &  a_{-\xi n} \le x \le 0 \\
0,  & \textrm{ otherwise }
\end{cases} \\
&\le&C\gamma_n(x).
\end{eqnarray*}
Thus, we have
\begin{eqnarray*}
\sum_{j=1}^n \frac{|\zeta_{jn}(x)|}{w(x_{j,n})}
[w(x_{j,n})|f(x_{j,n})-P_{n-1}^*(x_{j,n})|]
&\le& C E_{n-1}[f]_{w,\infty}\sum_{j=1}^n \frac{|\zeta_{jn}(x)|}{w(x_{j,n})} \\
&\le& C E_{n-1}[f]_{w,\infty}
\gamma_n(x).
\end{eqnarray*}
Consequently, we obtain using Theorem \ref{Theorem 2.6}
\begin{eqnarray*}
\lefteqn{|R_n[f;x]|} \\
&\le& |H_{w^2}[f-P_{n-1}^*;x]|
+ \sum_{j=1}^n \frac{|\zeta_{jn}(x)|}{w(x_{j,n})}
[w(x_{j,n})|f(x_{j,n})-P_{n-1}^*(x_{j,n})|] \\
&\le& |H_{w^2}[f-P_{n-1}^*;x]|
+ C E_{n-1}[f]_{w,\infty}\sum_{j=1}^n \frac{|\zeta_{jn}(x)|}{w(x_{j,n})} \\
&\le&  C_1 (1+  n^{-2}  w^{-1}(x))E_{n-1}[f]_{w,\infty} \log n
+ C_2 E_{n-1}[f]_{w,\infty}\gamma_n(x) \\
&\le& C\left\{ (1+  n^{-2}  w^{-1}(x))\log n + \gamma_n(x)\right\}E_{n-1}[f]_{w,\infty}.
\end{eqnarray*}
\end{proof}

\section{Estimation of the Functions of the Second Kind and Proof of Theorem \protect{\ref{Theorem 2.3}}}

In this section, we provide the finished proof of Theorem \ref{Theorem 2.3} 
and we shall prove upper bounds for the Chebyshev norms
of the functions of the second kind
\[
  q_n(x) := \cpv_{\mathbb R} \frac{p_n (t) w^2(t)}{t-x} dt,\quad x\in \mathbb R.
\]
Specifically, Criscuolo et al.\ \cite[Theorem 2.2(a)]{CDLM}
have shown that for large enough $n$,
\begin{equation}
  \label{eq:bound-qn}
  \| q_n \|_{L_\infty(\mathbb R)} \sim a_n^{-1/2}
\end{equation}
if $w^2$ is a symmetric weight of smooth polynomial decrease for large argument that satisfies some mild
additional smoothness conditions; cf.\ \cite[Definition 2.1]{CDLM} for
precise details. Our goal is to extend this result to a much larger
class of weight functions.

We start with an alternative representation for $q_n$. Here, again we use the notation
$\zeta_{jn}(.)$ for the weights of the Gaussian quadrature formula
with respect to the weight function $w^2$ associated to the node
$x_{jn}$.

\begin{lemma}
  \label{lem:51}
  We have the following identities for $n\geq 1$.
  \begin{enumerate}
  \item[(a)] If $x \ne x_{jn}$ for all $j$ then
    \[ q_n(x) = p_n(x) \left( \cpv_{\mathbb R} \frac{w^2(t)}{t-x} dt - \sum_{j=1}^n \frac{\zeta_{jn}(x)}{x_{jn} - x} \right). \]
  \item[(b)] For $j=1,2,\ldots,n$ we have
    \[ q_n(x_{jn}) = \zeta_{jn} p_n'(x_{jn}). \]
  \item[(c)] If $x \in (x_{\ell n}, x_{\ell-1,n})$ for some $2 \le \ell \le n$ then
    \[ |q_n(x)| \le |p_n(x)| \left( \sum_{j=\ell-1}^\ell \frac{\zeta_{jn}(x)}{|x - x_{jn}|}
                                    + \left| \cpv_{x_{\ell n}}^{x_{\ell-1,n}} \frac{w^2(t)}{x-t} dt \right| \right) . \]
  \item[(d)] If $x > x_{1n}$ then
    \[ |q_n(x)| \le  |p_n(x)| \left( \frac{\zeta_{1 n}}{|x - x_{1 n}|}
                                    + \left| \cpv_{x_{1 n}}^{d} \frac{w^2(t)}{x-t} dt \right| \right) . \]
  \item[(e)] If $x < x_{nn}$ then
    \[ |q_n(x)| \le  |p_n(x)| \left( \frac{\zeta_{n n}}{|x - x_{n n}|}
                                    + \left| \cpv_{c}^{x_{n n}} \frac{w^2(t)}{x-t} dt \right| \right) . \]
  \end{enumerate}
\end{lemma}

\begin{proof}
Parts (a), (b), (c), and (d) are shown for a special class of weights
in \cite[eqs.\ (5.2), (5.3), (5.4) and (5.5)]{CDLM}. An inspection of
the proof immediately reveals that no special properties of the weight
functions are ever used in these proofs and thus the exact same
methods of proof can be applied in our case. Part (e) can be shown by
arguments analog to those of the proof of (d).
\end{proof}

Our main result for this section then reads as follows.
\begin{theorem}\label{Theorem 3.2}
We have uniformly for $n$ large enough,
\begin{equation}\label{q_n}
\| q_n\|_{L_\infty(\mathbb R)}
\le C\frac1{\delta_n^{1/4}}\max\left\{\frac{T(a_n)}{a_n}, \frac{T(a_{-n})}{|a_{-n}|} \right\}^{1/4}\log n.
\end{equation}
\end{theorem}

\begin{proof}
We first consider the case that $x \ge a_{n/2}$. Then we write
\[ q_n(x) = \cpv_{\mathbb R} \frac{p_n (t) w^2(t)}{t-x} dt=I_1 + I_2 + I_3 \]
where $\varepsilon_n=a_{n/2}-a_{n/4}$ and
\begin{eqnarray*}
  I_1 &=& \int_{|t-x| \ge \varepsilon_n} \frac{p_n (t) w^2(t)}{t-x} dt, \\
  I_2 &=& \int_{n^{-10} < |t-x| < \varepsilon_n} \frac{p_n (t) w^2(t)}{t-x} dt, \\
  I_3 &=& \cpv_{x - n^{-10}}^{x + n^{-10}} \frac{p_n (t) w^2(t)}{t-x} dt.
\end{eqnarray*}
It then follows that
\begin{eqnarray*}
  |I_1|
    &\le& C\frac1{\varepsilon_n^{1/4}}  \int_{|t-x| \ge \varepsilon_n} |p_n (t) w^2(t)| dt \\
    &\le& C\frac1{\varepsilon_n^{1/4}}
    \left\|\left(\frac{w(t)}{|t-x|^{3/4}}\right)\right\|_{L_{\frac43}(|t-x| \ge \varepsilon_n)} \|p_nw\|_{L_4(I)} \\
    &\le& C \left(\frac{T(a_n)}{a_n}\right)^{1/4} (\log n)^{3/4} \delta_n^{-\frac14}(\log (n+1))^{\frac14}\\
    &\le& C \frac1{\delta_n^{1/4}}\max\left\{\frac{T(a_n)}{a_n}, \frac{T(a_{-n})}{|a_{-n}|} \right\}^{1/4}\log n.
\end{eqnarray*}
Here, we proceeded as follows: Since
\begin{eqnarray*}
\int_{|t-x| \ge \varepsilon_n} \frac{w^{4/3}(t)}{|t-x|} dt
&\le& \int_{|t-x| \ge 1 } \frac{w^{4/3}(t)}{|t-x|} dt
+ \int_{ \min\{\varepsilon_n, 1\} < |t-x| \le 1 } \frac{w^{4/3}(t)}{|t-x|} dt  \\
&\le& C_1 + \int_{ \min\{\varepsilon_n, 1\} < |t-x| \le 1 } \frac{1}{|t-x|} dt \\
&\le& C\log n,
\end{eqnarray*}
we see
\[
\left\|\left(\frac{w(t)}{|t-x|^{3/4}}\right)\right\|_{L_{\frac43}(|t-x| \ge \varepsilon_n)}
\le C (\log n)^{3/4}.
\]
Moreover,
\begin{eqnarray*}
  |I_2|
    & \le & \|p_n w\|_{L_\infty(\mathbb R)} w(a_{n/4}) \int_{n^{-10} < |t-x| < a_{n/4}} \frac{dt}{t-x} \\
    & \le & C \|p_n w\|_{L_\infty(\mathbb R)} w(a_{n/4}) \log n  \\
    & \le & C \frac1{\delta_n^{1/4}}\max\left\{\frac{T(a_n)}{a_n}, \frac{T(a_{-n})}{|a_{-n}|} \right\}^{1/4}\log n.
\end{eqnarray*}
in view of the decay behaviour of $w(a_{n/4})$.
Finally,
\begin{eqnarray*}
  |I_3|
    & = & \left| \cpv_{x - n^{-10}}^{x + n^{-10}} \frac{p_n (t) w^2(t) - p_n (x) w^2(x)}{t-x} dt \right| \\
    & \le & 2 n^{-10} \| (p_n w^2)' \|_{L_\infty(I)} \\
    & \le&C\frac1{\delta_n^{1/4}}\max\left\{\frac{T(a_n)}{a_n}, \frac{T(a_{-n})}{|a_{-n}|} \right\}^{1/4}\log n,
\end{eqnarray*}
using $\|Q'w\| < \infty$,
and $\left|(p_nw^2)'\right| \le C \left( \left|p_n'(x)w(x)\right| +\left|p_n(x)w(x)\right|\right)$.

Thus we conclude that for  $x \ge a_{n/2}$
\[
  |q_n(x)| \le C \frac1{\delta_n^{1/4}}\max\left\{\frac{T(a_n)}{a_n}, \frac{T(a_{-n})}{|a_{-n}|} \right\}^{1/4}\log n.
\]
The bound for $x \le a_{-n/2}$ follows by an analog argument.

It remains to show the required inequality for $a_{-n/2}\le x \le a_{n/2}$. We split up this case into a number of sub-cases.
First, if $x = x_{jn}$ for some $j$, then we obtain from Lemma \ref{Lemma 2.3}(g) and Lemma \ref{Lemma 2.4} that
\begin{eqnarray*}
  q_n(x_{jn}) &=& \zeta_{jn}(x_{jn}) p_n'(x_{jn}) \\
         & \sim & |(x_{j,n}-a_{-n})(a_n-x_{j,n})|^{-1/4}.
\end{eqnarray*}
In the remaining case, $x$ does not coincide with any of the zeros of
the orthogonal polynomial $p_n$. We only treat the case $x \ge 0$
explicitly; the case $x<0$ can be handled in a similar fashion.
In this situation, we have that $x \in (x_{\ell n}, x_{\ell-1, n})$
for some $\ell \in \{2,3,\ldots,n\}$, and therefore we may invoke the
representation from Lemma \ref{lem:51}(c). This yields
\[
   |q_n(x)| \le  \sum_{j=\ell-1}^\ell \frac{\zeta_{jn}(x)|p_n(x)|}{|x - x_{jn}|}
                                    + \left| p_n(x) \cpv_{x_{\ell n}}^{x_{\ell-1,n}} \frac{w^2(t)}{x-t} dt \right|
\]
Here we first look at the two terms inside the summation
operator. From Lemma \ref{Lemma 2.3} (i), Lemma \ref{Lemma 2.4} and Lemma \ref{Lemma 2.3}(j)
we have for $x \in (x_{j+1, n}, x_{j, n})$,
\begin{eqnarray*}
  \label{eq:qn1}
\frac{\zeta_{jn}(x)|p_n(x)|}{|x - x_{jn}|}
&\le& C w^{-1}(x)w^2(x_{j,n})|(x_{j,n}-a_{-n})(a_n-x_{j,n})|^{-1/4}  \\
&\le& C\frac1{\delta_n^{1/4}}\max\left\{\frac{T(a_n)}{a_n}, \frac{T(a_{-n})}{|a_{-n}|} \right\}^{1/4}.
\end{eqnarray*}
Moreover, the remaining term can be bounded as follows. We define
\[
  h := \min \left\{ \frac{1}{n}\sqrt{\frac{\delta_n a_n}{T(a_n)}} , x - x_{\ell n}, x_{\ell-1,n} - x \right\}
\]
and write
\[
  \cpv_{x_{\ell n}}^{x_{\ell-1,n}} \frac{w^2(t)}{x-t} dt = I_4 + I_5 + I_6
\]
where
\begin{eqnarray*}
  I_4 &=& \int_{x_{\ell n}}^{x-h} \frac{w^2(t)}{x-t} dt, \\
  I_5 &=& \cpv_{x-h}^{x+h} \frac{w^2(t)}{x-t} dt, \\
  I_6 &=& \int_{x + h}^{x_{\ell-1,n}} \frac{w^2(t)}{x-t} dt.
\end{eqnarray*}
Looking at $I_5$ first and using the definition of $h$ and the monotonicity
properties of $w$ and $Q'$, we obtain
\begin{eqnarray*}
  |I_5|
    & = & \left| \int_{x-h}^{x+h} \frac{w^2(t) - w^2(x)}{x-t} dt \right|
     \le  2 h \|(w^2)'\|_{L_\infty[x-h,x+h]} \\
    & \le & 4 h w^2(x-h) Q'(x+h) \le C w^2(x_{\ell n}),
\end{eqnarray*}
because
\[
|h Q'(x+h)| \le C \frac{1}{n}\sqrt{\frac{\delta_n a_n}{T(a_n)}}Q'(a_n) \le C.
\]
Thus, we see by Lemma \ref{Lemma 2.3}(b)
\[
  |p_n(x)| \cdot |I_5|
    \le C w^2(x_{\ell n}) w^{-1}(x) |(x-a_{-n})(a_n-x)|^{-1/4}.
\]
Moreover, we know that our function $w$ is decreasing in $(0,\infty)$
and satisfies $w(x) \sim 1$ whenever $x$ is confined to a fixed finite
interval. Thus,
\begin{eqnarray}
  \nonumber
  |I_4|
    &\le& C w^2(x_{\ell n})  \int_{x_{\ell n}}^{x-h} \frac{dt}{x-t}
     =  C w^2(x_{\ell n}) \log \frac{x-x_{\ell n}}{h} \\
  \label{eq:i4a}
    &\le& C w^2(x_{\ell n}) \log \frac{x_{\ell-1,n}-x_{\ell n}}{h}.
\end{eqnarray}
Another estimate for the quantity $|I_4|$ will also be useful later:
We can see that
\begin{equation}
  \label{eq:i4b}
  |I_4|
     \le \frac1h \int_{x_{\ell n}}^d w^2(t) dt
     \le \frac1{2 h Q'(x_{\ell n})} \int_{x_{\ell n}}^d 2 Q'(t) w^2(t) dt
     =   \frac{w^2(x_{\ell n})}{2h Q'(x_{\ell n})}.
\end{equation}
Using essentially the same arguments, we can provide corresponding
bounds for $|I_6|$, viz.
\begin{equation}
  \label{eq:i6a}
  |I_6| \le C w^2(x_{\ell n}) \log \frac{x_{\ell-1,n}-x_{\ell n}}{h}
\end{equation}
and
\begin{equation}
  \label{eq:i6b}
  |I_6| \le  \frac{w^2(x_{\ell n})}{2h Q'(x_{\ell n})}.
\end{equation}
Now we recall that $h$ was defined as the minimum of three quantities
and we check with which of these quantities it coincides.
\begin{itemize}
\item  If $h = \frac{1}{n}\sqrt{\frac{\delta_n a_n}{T(a_n)}}$ then we use eqs.\ (\ref{eq:i4a}), (\ref{eq:i6a}),
Lemma \ref{Lemma 2.3} (b) and Lemma \ref{Lemma 2.3}(j) to obtain
  \begin{eqnarray*}
    (|I_4| + |I_6|) \cdot |p_n(x)|
      & \le & C w^2(x_{\ell n}) |p_n(x)|
              \log \left\{(x_{\ell-1,n}-x_{\ell n})n\sqrt{\frac{T(a_n)}{\delta_n a_n}}\right\} \\
      & \le & C w^2(x_{\ell n}) |p_n(x)|\log (T(a_n)) \\
      & \le & C  w^2(x_{\ell n}) w^{-1}(x)|(x-a_{-n})(a_n-x)|^{-1/4}\log (T(a_n))\\
      & \le & C  w^2(x_{\ell n}) w^{-1}(x)\frac1{\delta_n^{1/4}}\max\left\{\frac{T(a_n)}{a_n}, \frac{T(a_{-n})}{|a_{-n}|} \right\}^{1/4}\log n
  \end{eqnarray*}
because  we see from Lemma \ref{Lemma 2.3}(h)
\[
(x_{\ell-1,n}-x_{\ell n}) \sim \varphi_n(x) \le C\frac{\sqrt{\delta_n a_nT(a_n)}}{n}.
\]
\item If $h = x - x_{\ell n}$ then we require the known bound
  \[
    \left| \frac{p_n(x) w(x)}{x - x_{jn}}\right|
    \le C \varphi_n^{-1}(x_{j,n})|(x_{j,n}-a_{-n})(a_n-x_{j,n})|^{-1/4}.
  \]
  that holds uniformly for all $j$ and all $x \in \mathbb R$
   to derive that
  \begin{eqnarray*}
    \lefteqn{(|I_4| + |I_6|) \cdot |p_n(x)| } \\
      & \le & C w^2(x_{\ell n}) (x - x_{\ell n}) w^{-1}(x)
      \varphi_n^{-1}(x_{j,n}) \log \frac{x_{\ell-1,n} - x_{\ell n}}{x - x_{\ell n}} \\
      && \quad \times |(x_{j,n}-a_{-n})(a_n-x_{j,n})|^{-1/4}\\
             & \le & C w^2(x_{\ell n}) \frac{x - x_{\ell n}}{x_{\ell-1,n} - x_{\ell n}} w^{-1}(x)
         \log \frac{x_{\ell-1,n} - x_{\ell n}}{x - x_{\ell n}} \\
         && \quad \times |(x_{j,n}-a_{-n})(a_n-x_{j,n})|^{-1/4}\\
      & \le & C w^2(x_{\ell n}) w^{-1}(x)
      |(x_{j,n}-a_{-n})(a_n-x_{j,n})|^{-1/4}\\
       & \le & C  w^2(x_{\ell n}) w^{-1}(x)\frac1{\delta_n^{1/4}}\max\left\{\frac{T(a_n)}{a_n}, \frac{T(a_{-n})}{|a_{-n}|} \right\}^{1/4}\log n.
  \end{eqnarray*}
  where we used Lemma \ref{Lemma 2.3}(h) about the spacing of the nodes $x_{jn}$ and Lemma \ref{Lemma 2.3}(j).
\item The final case $h = x_{\ell-1,n} - x$ is essentially the same as
  the previous one and leads to the same bounds.
\end{itemize}

Combining eq.\ (\ref{eq:qn1}) with the estimates for $I_4$, $I_5$ and
$I_6$ we thus obtain, for our range of $x$,
\begin{eqnarray*}
  |q_n(x)|
   & \le & C  w^2(x_{\ell n}) w^{-1}(x)\frac1{\delta_n^{1/4}}\max\left\{\frac{T(a_n)}{a_n}, \frac{T(a_{-n})}{|a_{-n}|} \right\}^{1/4}\log n\\
    & \le & C  \frac1{\delta_n^{1/4}}\max\left\{\frac{T(a_n)}{a_n}, \frac{T(a_{-n})}{|a_{-n}|} \right\}^{1/4}\log n.
\end{eqnarray*}
\end{proof}

\begin{proof}[Proof of Theorem \protect{\ref{Theorem 2.3}}]
Lemma \ref{Lemma 2.4} and Theorem \ref{Theorem 1.3_1} imply the result.
\end{proof}

\begin{remark}
We guess that the factor of $\ln n$ may be made smaller in the upper bound above for large enough $n$ when $T\sim 1$, i.e., when $Q$ is of smooth polynomial growth for large argument. Indeed, the sequence $\varepsilon_n$ then grows without bound uniformly for large enough $n$. When $T$ grows without bound for large argument, i.e., when $Q$ is of smooth faster than polynomial growth for large argument, this is not the case. See Lemma \ref{Lemma 2.3} (f3) and the proof of Theorem \ref{Theorem 2.4}.
\end{remark}

\setcounter{equation}{0}
\section{Numerical examples}
\label{sec:num-examples}

In this section, we provide some numerical results to illustrate our
theoretical findings. As the algorithmic aspects regarding the concrete implementation
are not within the main focus of this paper, we have relegated the discussion of
such details to Appendix \ref{sec:num-comments}.

In all our examples, we have chosen the Freud-type weight function $w = \exp(-Q)$ with
the external field $Q(t) = t^4$.
	
\begin{example}
	\label{ex:1}
	 The first example deals with the function $f(t) = \sin t$. We have computed 
	 the values of $H_{w^2}[f;x]$ numerically for $x \in \{ 0.01, 0.1, 0.5, 1, 2 \}$.
	 The approximation was done with our algorithm with $n$ nodes, where $n \in \{1, 2, 3, \ldots, 30 \}.$
	 In Figures \ref{fig:sin1}--\ref{fig:sin5}, we plot the associated absolute errors versus the number of
	 nodes. All plots have a logarithmic scale on the vertical axis. Note that not all values of~$n$
	 are included in all plots. This is due to the fact that, for larger values of $n$, the relative error was 
	 smaller than machine accuracy, so the errors were effectively zero in these cases, which precludes
	 their inclusion into a logarithmic plot. The rapid (essentially exponential)
	 decay of the errors as $n$ increases is clearly evident.
\end{example}

\begin{figure}[h]
	\centering
	\includegraphics[width=0.6\textwidth]{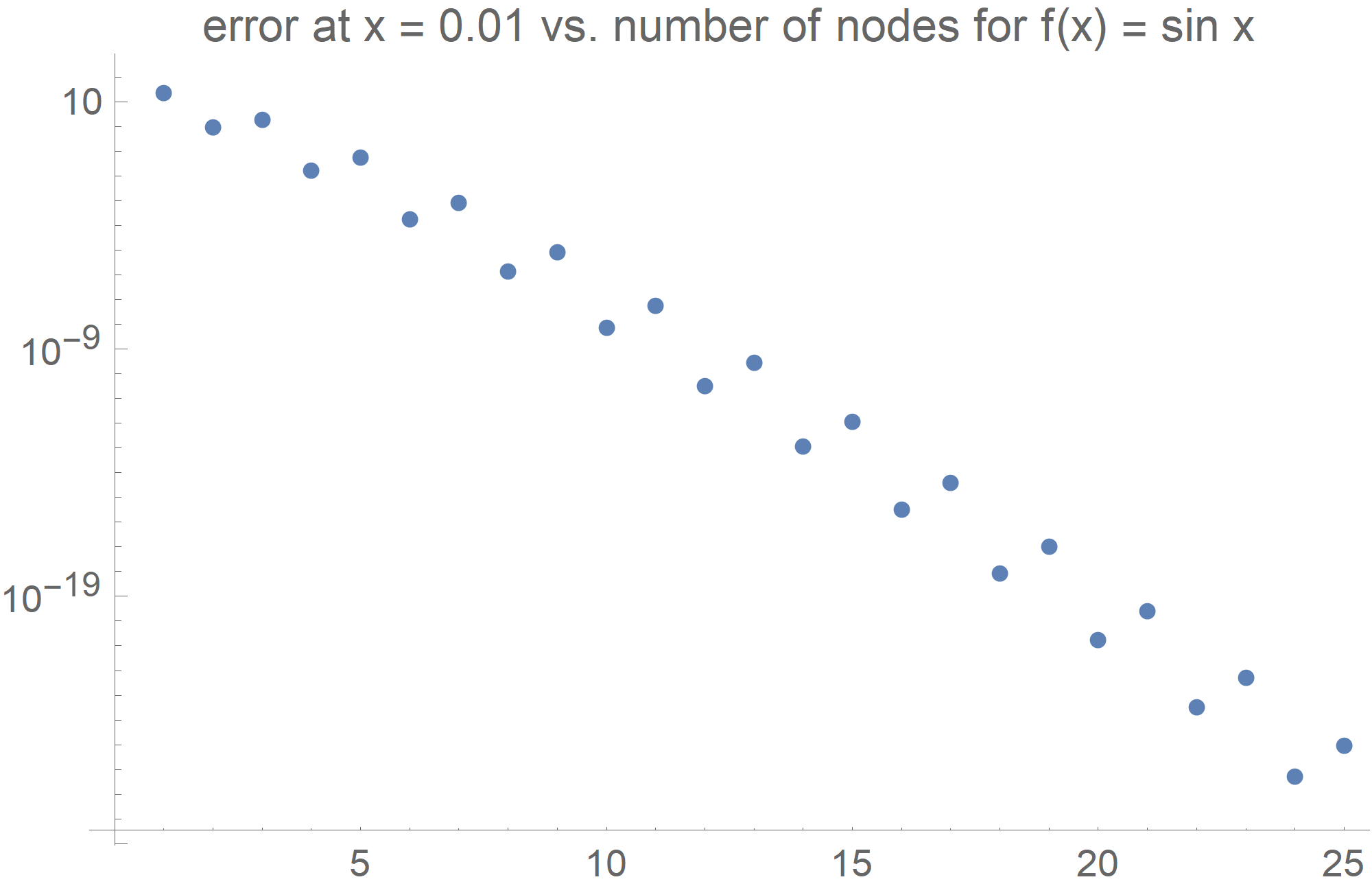}
	\caption{\label{fig:sin1}Absolute errors of $H_{w^2}[f;x]$ when computed by our algorithm 
		vs.\ the number $n$ of nodes for $f(t) = \sin t$ and $x = 0.01$.}
\end{figure}

\begin{figure}[h]
	\centering
	\includegraphics[width=0.6\textwidth]{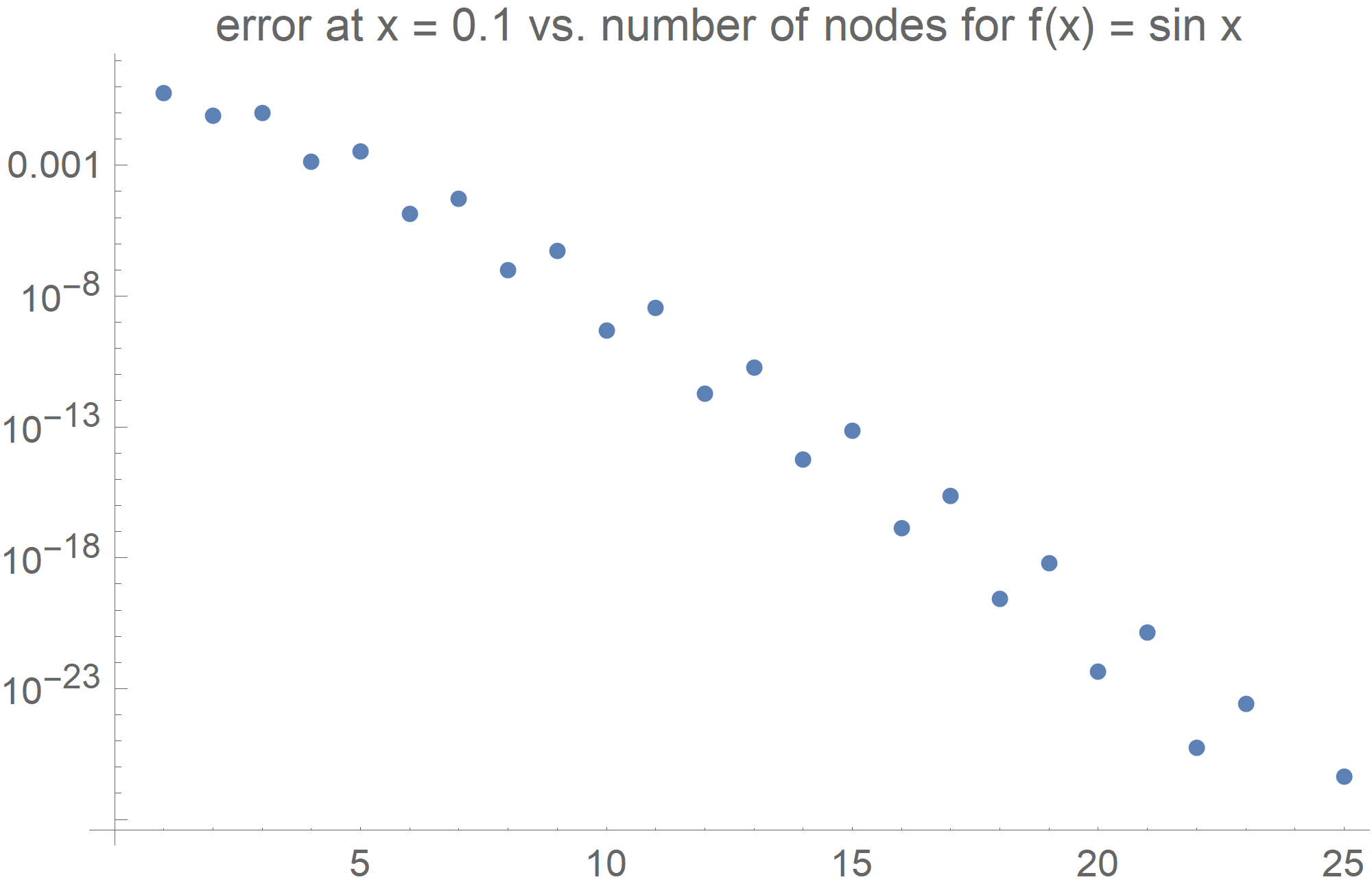}
	\caption{\label{fig:sin2}Absolute errors of $H_{w^2}[f;x]$ when computed by our algorithm 
		vs.\ the number $n$ of nodes for $f(t) = \sin t$ and $x = 0.1$.}
\end{figure}

\begin{figure}[h]
	\centering
	\includegraphics[width=0.6\textwidth]{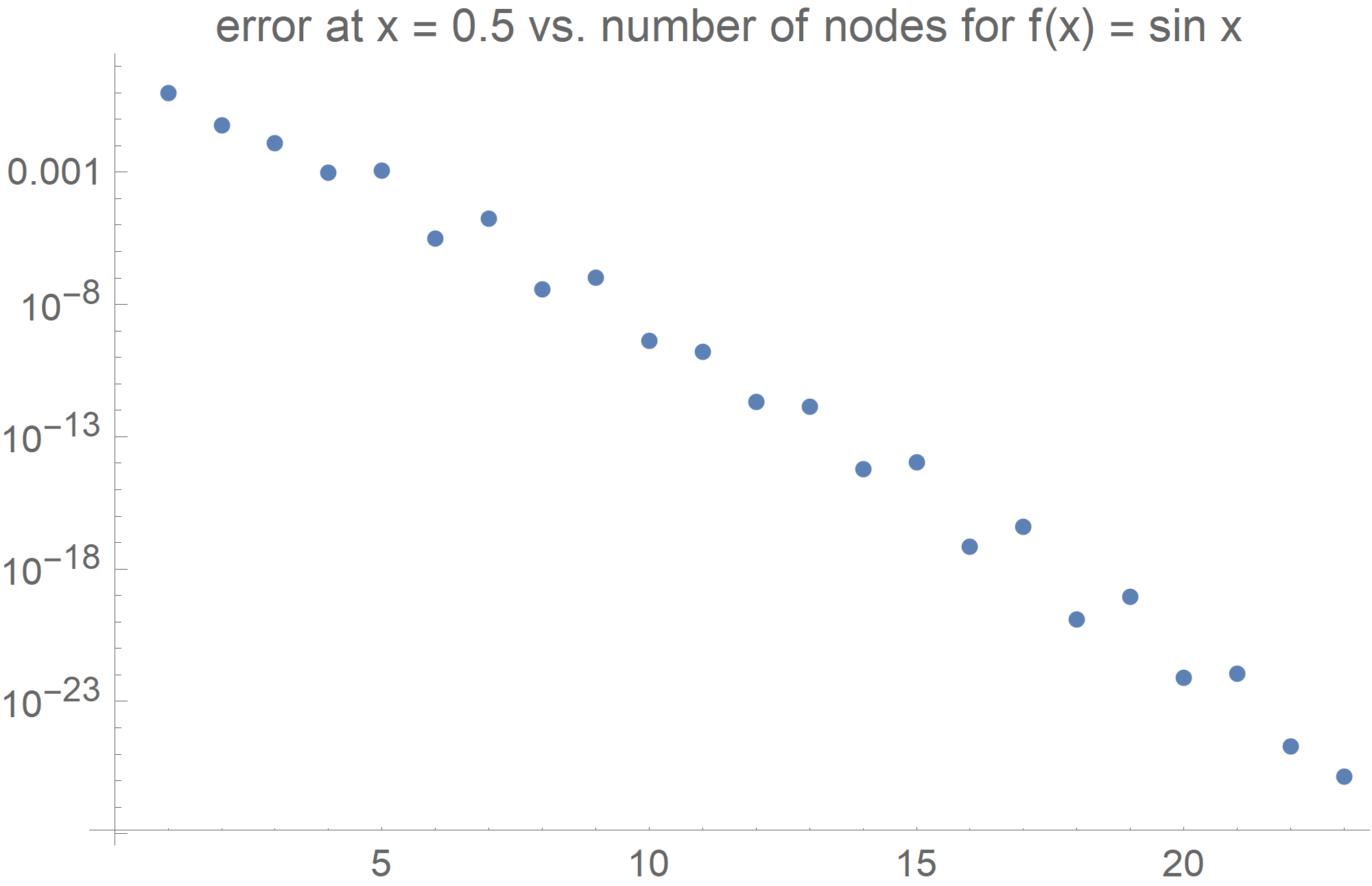}
	\caption{\label{fig:sin3}Absolute errors of $H_{w^2}[f;x]$ when computed by our algorithm 
		vs.\ the number $n$ of nodes for $f(t) = \sin t$ and $x = 0.5$.}
\end{figure}

\begin{figure}[h]
	\centering
	\includegraphics[width=0.6\textwidth]{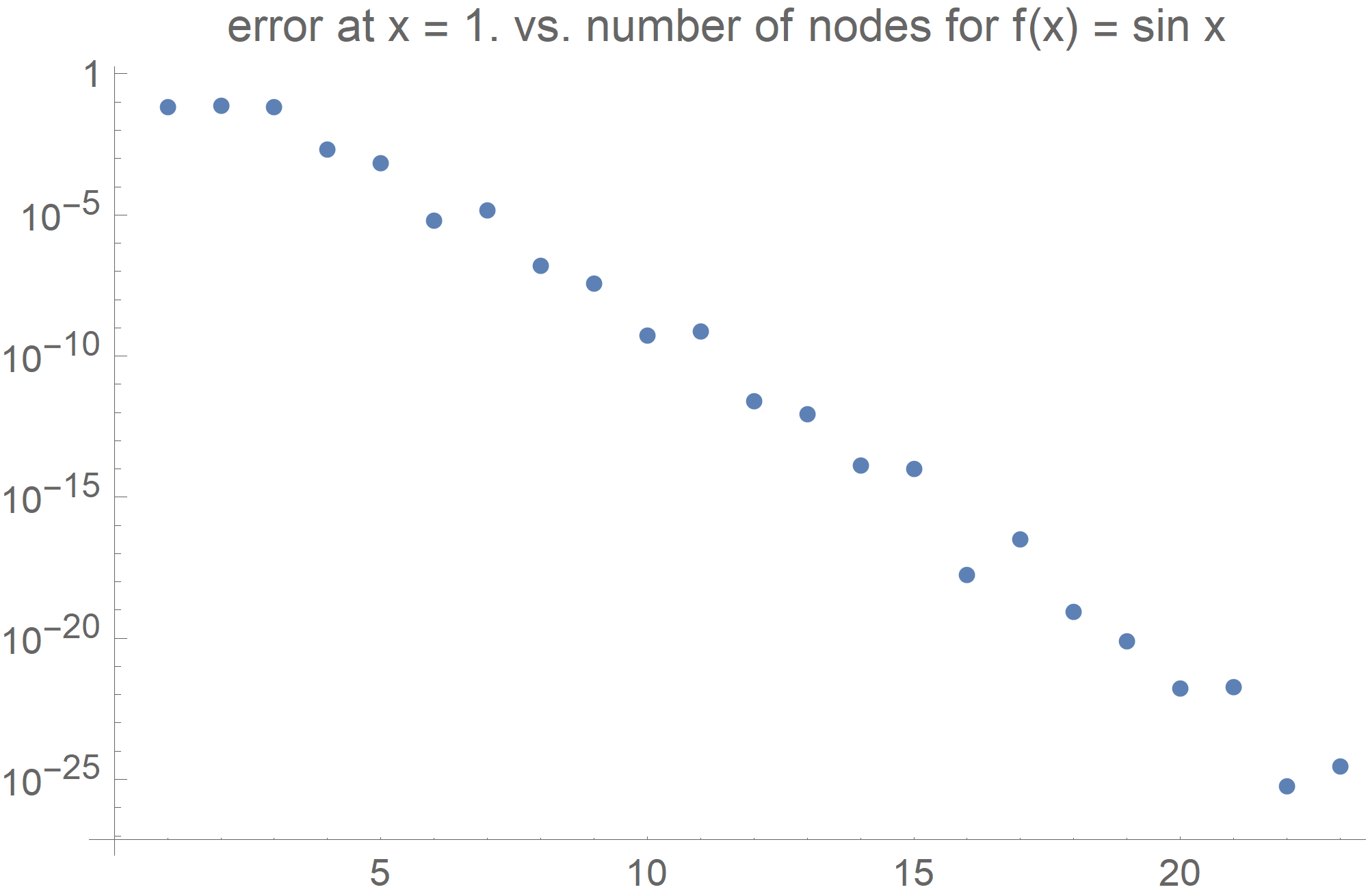}
	\caption{\label{fig:sin4}Absolute errors of $H_{w^2}[f;x]$ when computed by our algorithm 
		vs.\ the number $n$ of nodes for $f(t) = \sin t$ and $x = 1$.}
\end{figure}

\begin{figure}[h]
	\centering
	\includegraphics[width=0.6\textwidth]{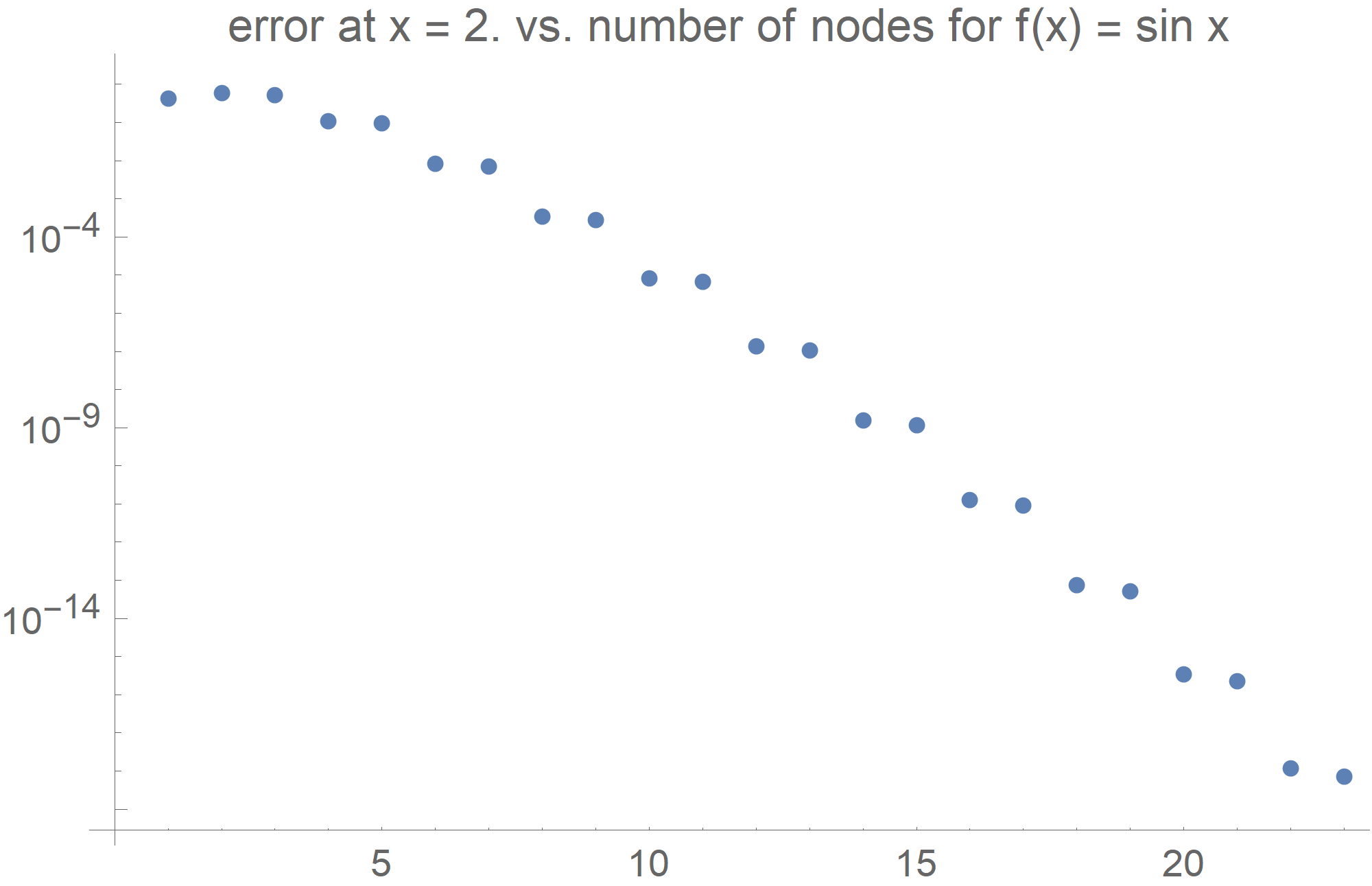}
	\caption{\label{fig:sin5}Absolute errors of $H_{w^2}[f;x]$ when computed by our algorithm 
		vs.\ the number $n$ of nodes for $f(t) = \sin t$ and $x = 2$.}
\end{figure}

\clearpage

\begin{example}
	 The second example is very similar to the first one, the only change being that we now use the
	 function $f(t) = \log(1+t^2)$. 
	 For this example, we can observe the same very good convergence properties as in
	 Example \ref{ex:1}, see Figures \ref{fig:log1}--\ref{fig:log5}.
\end{example}

\begin{figure}[h]
	\centering
	\includegraphics[width=0.6\textwidth]{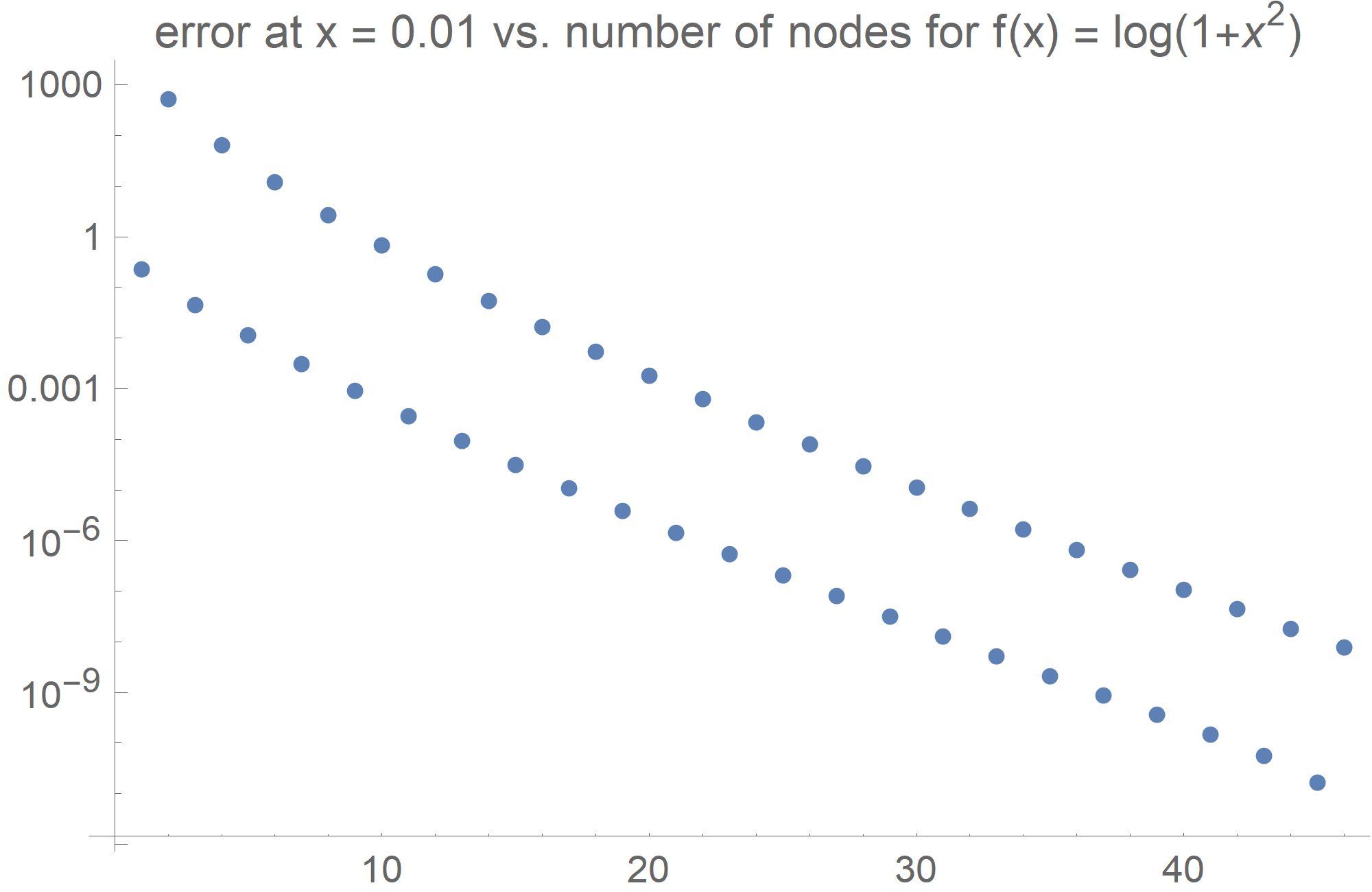}
	\caption{\label{fig:log1}Absolute errors of $H_{w^2}[f;x]$ when computed by our algorithm 
		vs.\ the number $n$ of nodes for $f(t) = \log(1 + t^2)$ and $x = 0.01$.}
\end{figure}

\begin{figure}[h]
	\centering
	\includegraphics[width=0.6\textwidth]{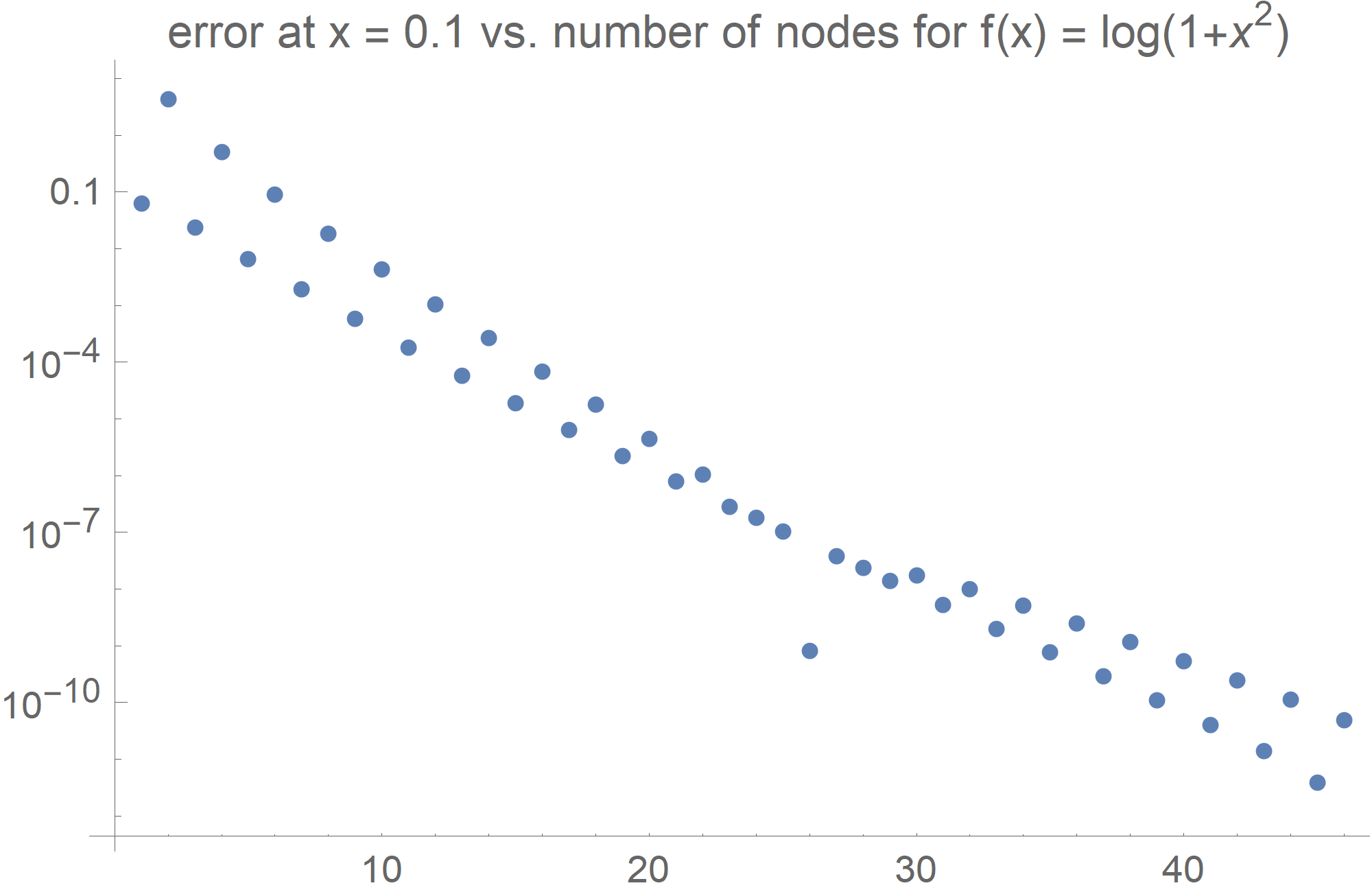}
	\caption{\label{fig:log2}Absolute errors of $H_{w^2}[f;x]$ when computed by our algorithm 
		vs.\ the number $n$ of nodes for $f(t) = \log(1 + t^2)$ and $x = 0.1$.}
\end{figure}

\begin{figure}[h]
	\centering
	\includegraphics[width=0.6\textwidth]{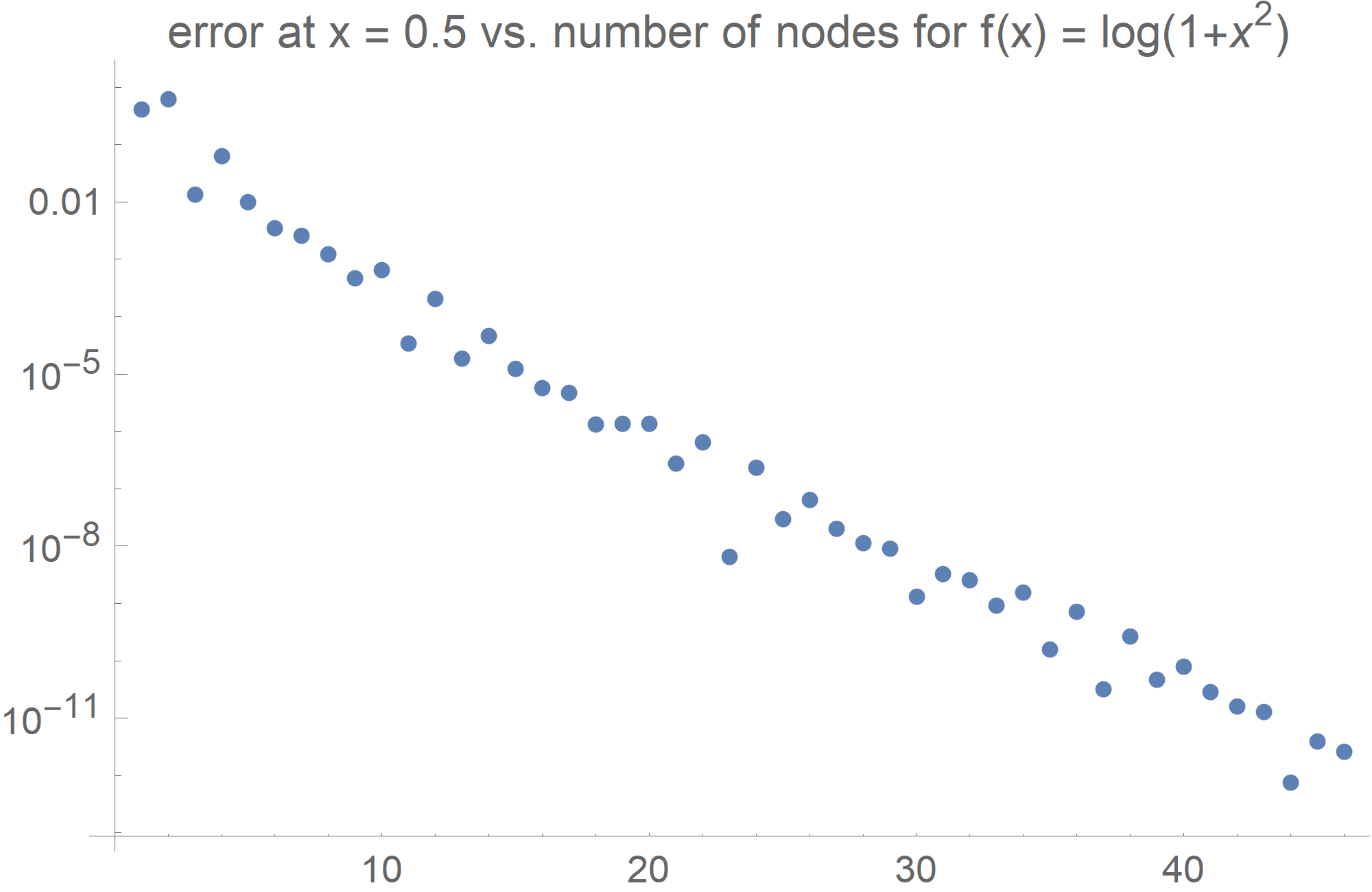}
	\caption{\label{fig:log3}Absolute errors of $H_{w^2}[f;x]$ when computed by our algorithm 
		vs.\ the number $n$ of nodes for $f(t) = \log(1 + t^2)$ and $x = 0.5$.}
\end{figure}

\begin{figure}[h]
	\centering
	\includegraphics[width=0.6\textwidth]{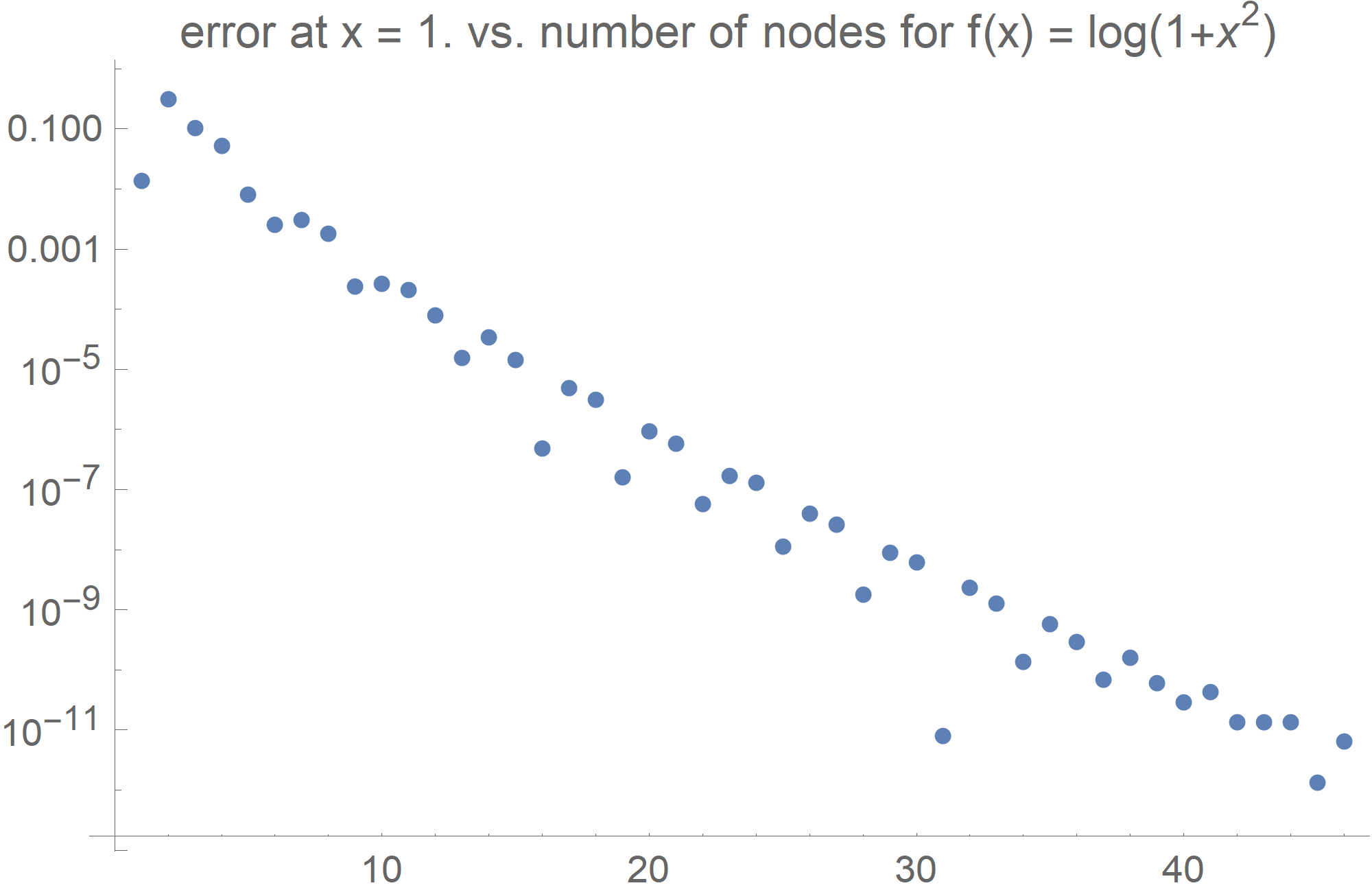}
	\caption{\label{fig:log4}Absolute errors of $H_{w^2}[f;x]$ when computed by our algorithm 
		vs.\ the number $n$ of nodes for $f(t) = \log(1 + t^2)$ and $x = 1$.}
\end{figure}

\begin{figure}[h]
	\centering
	\includegraphics[width=0.6\textwidth]{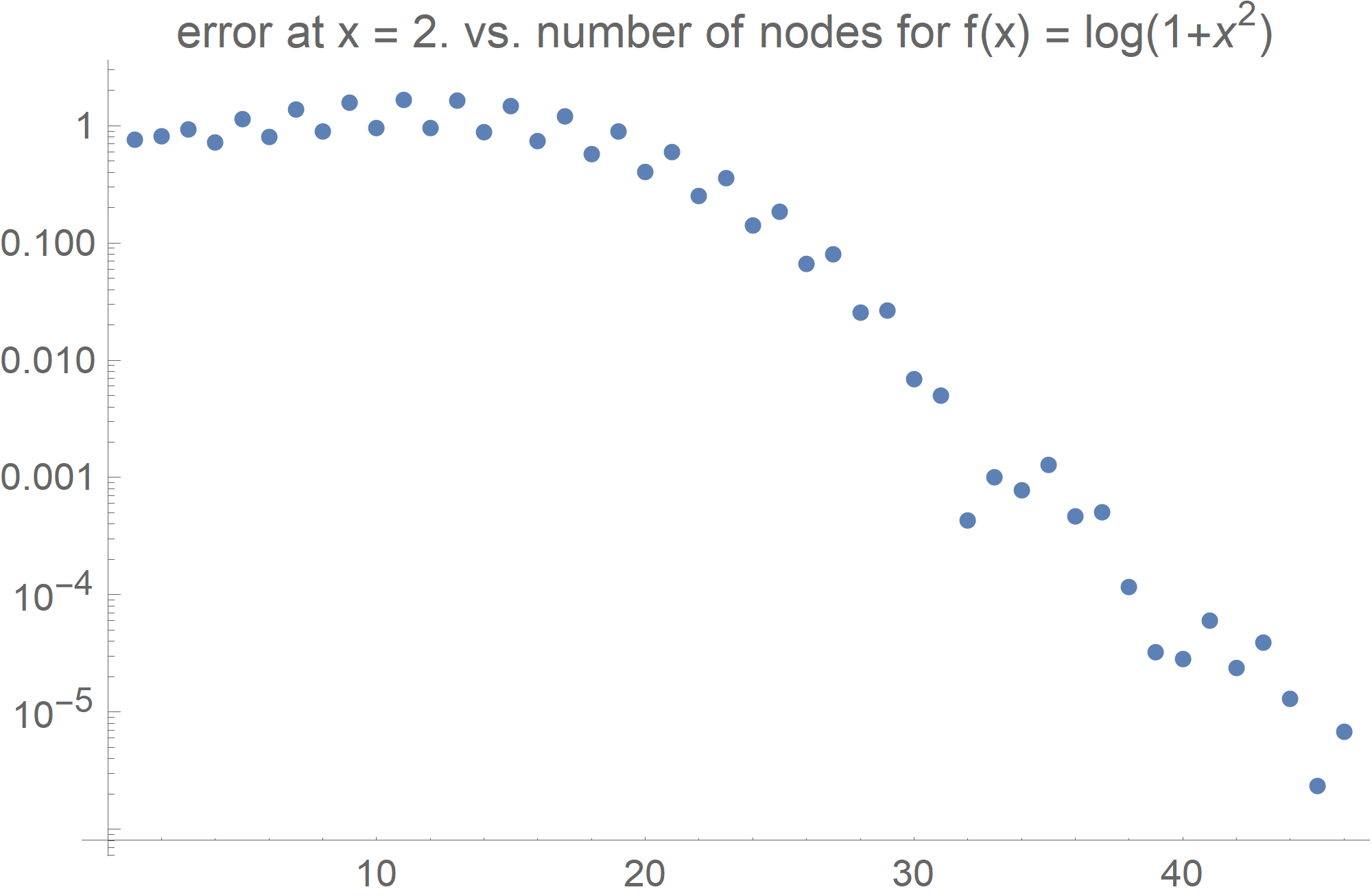}
	\caption{\label{fig:log5}Absolute errors of $H_{w^2}[f;x]$ when computed by our algorithm 
		vs.\ the number $n$ of nodes for $f(t) = \log(1 + t^2)$ and $x = 2$.}
\end{figure}

\clearpage

\appendix

\setcounter{equation}{0}
\section{Some needed potential theoretical background}

In order to understand the deep differences in the problems we consider in this paper to those studied in our papers \cite{DD1,DD2,DD3,DD4} and the complexities involved from moving from the results of \cite{DD1,DD2,DD3,DD4} to those in this paper, we believe it useful to add the following last section as an appendix. The  fundamental weighted energy problem on the real line:
\medskip

Let $\Sigma$ be a closed set on the real line and $w=\exp(-Q):\Sigma\to [0,\infty) $ be an upper semi-continuous weight function on $\Sigma$ with external field $Q$ that is positive on a set of positive linear Lebesgue measure. If $\Sigma$ is unbounded, we assume that 
\[
\lim_{|x|\to \infty,\, x\in \Sigma}(Q(x)-\log(x))=\infty.
\]
Fix $\Sigma$ and $Q$ and consider
\[
{\rm inf}_{\mu}\left(\int\int \log\frac{1}{|x-t|}d\mu(x)d\mu(t)+2\int Qd\mu\right)
\]
where the infimum is taken over all positive Borel measures $\mu$ with the support of $\mu$, ${\rm supp}(\mu)$ in $\Sigma$ and with $\mu(\Sigma)=1$. The infimum is attained by a unique minimizer $\mu_w$. Let
\[
V^{\mu_w}(x)=\int \log\left(\frac{1}{|x-t|}\right)d\mu_w(t),\, x\in \mathbb R
\]
be the logarithmic potential for $\mu_w$. Then the following variational inequalities hold:
\begin{equation}
\begin{array}{ccc}
V^{\mu_w}(x)+Q(x)\geq F_w & {\rm q.e.} & x\in\Sigma \\ 
V^{\mu_w}(x)+Q(x)=F_w & {\rm q.e.} & x\in {\rm supp}(\mu_w)
\end{array}
\end{equation}
Here, $F_w$ is a constant and q.e.\ (quasi everywhere) means, with the exception of a set of logarithmic capacity zero.

The support of $\mu_w$ is one of the most important and fundamental quantities to determine the minimizer $\mu_w$, an extremely important and challenging problem which appears   in diverse areas in mathematics and physics such as orthogonal polynomials, random matrix theory, combinatorics, approximation theory, electron configiurations on conductors, integrable systems, number theory and many more. When $Q$ is identically zero,  the support of $\mu_w$ typically "lives" close to the boundary of the set $\Sigma$ but when $Q$ is no longer identically zero, the support of $\mu_w$ depends heavily on $Q$ for example, its regularity and smoothness and can be quite arbitrary. The more complicated 
support of $\mu_w$ in the case of the work in this paper compared to the support of $\mu_w$ in our papers \cite{DD1,DD2,DD3,DD4}, is one important reason why the research in the current paper differs from our previous work in \cite{DD1,DD2,DD3,DD4} so substantially. 

\subsection{The case $\Sigma=\mathbb R$ and one interval: $Q$ even.} 

The research on orthogonal polynomials, their zeroes and associated Christoffel functions, as used as critical tools in our papers \cite{DD1,DD2,DD3,DD4}, was developed  by Lubinsky and Levin \cite{LL}. In particular, if 
$Q$ is even, $Q'$ exists in $(0,\infty)$ and $xQ'$ increasing on $(0,\infty)$ and positive there, the support of $\mu_w$ is given in Remark 2.3.
 
\subsection{The case $\Sigma=\mathbb R$ and one interval.}

Here, Lubinsky and Levin in their classic monograph \cite{LL}, established remarkable research on orthogonal polynomials, their zeroes and Christoffel functions to allow for the research in this paper. See Sections (2-5). In particular, under hypotheses on $Q$ such as given in Definition 2.1, the support of $\mu_w$ is given as in Definition 2.2.

\subsection{The case $\Sigma=\mathbb R$}

Deift and his collaborators in their papers \cite{Deift, Deift1, Deift2} studied the support of $\mu_w$ for smooth $Q$, for example polynomials and obtained many term asymptotics for the associated orthogonal polynomials and their zeroes. We did not use their research in this paper and leave that for future work. In this case, the support of $\mu_w$ typically need not be one interval; indeed it often splits into a finite number of intervals (sometimes with gaps) with endpoints described often using tools such as Riemann Hilbert problems.

\subsection{Some other cases.}

Damelin, Benko, Dragnev, Kuijlaars, Deift, Olver \cite{S1,S2,S3,O,O1} and many others have studied cases of $Q$ and $\Sigma$ where the support $\mu_w$ splits into a finite number of intervals (often with gaps) and with endpoints not necessarily known. Lubinsky and Levin have in recent years established remarkable results on asymptotics of orthogonal polynomials, their zeroes and Christoffel functions under very mild conditions on $Q$ and on various sets $\Sigma$. We do not use their research in this paper. See \cite{LL1,LL2}.
\medskip

In summary, descriptions of supports of minimizers for logarithmic energy variational problems such as (6.1) (and we do not discuss other kernels!) and associated research on their orthogonal polynomials, zeros and Christoffel functions for different $Q$
is a huge area of research in many areas of mathematics and physics. In particular, and in this regard, our results and methods in this paper, generalize, highly non-trivially our work in \cite{DD1,DD2,DD3,DD4}.

\setcounter{equation}{0}
\section{Comments on the Numerical Method}
\label{sec:num-comments}

In this appendix we collect some information that is helpful in the construction and implementation
of the numerical algorithm for the quadrature formula \eqref{w_jn} required for the numerical examples
presented in Section \ref{sec:num-examples}. To avoid excessive technical complications, the discussion
here will not cover the very general weight functions investigated in the main part of the paper. 
Rather, we will restrict our attention to
the special case discussed in Section \ref{sec:num-examples}, i.e.\ we shall assume throughout this appendix
that $w$ is the Freud-type weight given by
\begin{equation}
	\label{eq:wnum}
	w(t) = \exp(-Q(t)) 
	\qquad \mbox{ with } \qquad
	Q(t) = t^4
\end{equation}
for $t \in \mathbb R$.
Furthermore, as indicated in Section \ref{sec:num-examples}, the algorithmic aspects are not the main
point of this more theoretically oriented paper. Therefore, we emphasize here
that the description below is also rather theoretical and does not include issues like
the numerical stability of the approach. It is well known that this is a highly 
nontrivial matter that, however, needs to be discussed elsewhere.

The main observation in the present context is that, in view of eq.\ \eqref{eq:qf-hl}, 
the construction of the formula $Q_n[f;x]$ requires the following ingredients:
\begin{enumerate}
\item We need to be able to compute the Lagrange interpolation polynomial $L_n[f]$ for the given 
	function $f$. In view of the well known general relation
	\begin{equation}
		\label{eq:lagrange}
		L_n[f] (t)
		= \sum_{j=1}^n f(x_{j,n}) 
			\prod_{\stackrel{\scriptstyle k=1}{k\ne j}}^n
				\frac{t - x_{k,n}}{x_{j,n} - x_{k,n}},
	\end{equation}
	this means that we need to know the location of the nodes $x_{j,n}$ ($j = 1, 2, \ldots, n$),
	i.e.\ the zeros of the orthonormal polynomials $p_n = p_n(w^2, \cdot)$ with respect to the
	weight function $w^2$.
\item In the second step, it is necessary to apply the weighted Hilbert transform operator $H_{w^2}$
	to this interpolation polynomial. In view of the linearity of the Hilbert transform, this demands the 
	knowledge of the values of $H_{w^2}[\pi_k; x]$ for $k = 0, 1, 2, \ldots, n$ where
	$\pi_k(t) = t^k$ is the $k$th monomial.
\end{enumerate}
We shall now describe how this information can be obtained. 

\subsection{The moments of the weight function $w^2$}

It turns out that, for both required items, it is necessary to compute the moments 
\begin{equation}
	\label{eq:def-moments}
	\mu_k := \int_{-\infty}^\infty w^2(t) t^k dt
\end{equation}
of the given weight function $w^2$ for $k = 0, 1, 2, \ldots, 2n$, so this is our first result. Indeed, 
we can see that
\begin{equation}
	\label{eq:moments}
	\mu_k = \begin{cases}
		0 & \mbox{ if } k \mbox{ is odd,} \\
		2^{-(k+5)/4} \cdot \Gamma\left(\frac{k + 1}4 \right) & \mbox{ if } k \mbox{ is even,}
		\end {cases}
\end{equation}
where $\Gamma$ denotes Euler's Gamma function.
The result for odd values of $k$ immediately follows from a symmetry argument because the weight
function $w^2$ is even, 
and the result for even values of $k$ can be obtained by a symbolic integration using 
a computer algebra package like Mathematica \cite{Math}.

\subsection{The nodes of the interpolation operator $L_n$}

As indicated above, the nodes of the interpolation operator $L_n$ are the zeros of the
orthonormal polynomial $p_n(w^2, \cdot)$. To determine these values, we follow a strategy outlined 
in \cite{Gau}. Specifically, we consider the orthogonal polynomials
$\tilde p_n(w^2, \cdot)$ for the weight function $w^2$ that, instead of being normalized according to 
\eqref{eq:onp}, are normalized such that their leading coefficient is 1. Clearly, this means that, 
for each $n$, there exists some real number $C_n$ such $p_n(w^2, x) = C_n \tilde p_n(w^2, x)$ holds for
all $x \in \mathbb R$, and hence the zeros of $p_n(w^2, \cdot)$ coincide with those of $\tilde p_n(w^2, \cdot)$.
It is then a well known general property of orthogonal polynomials that there exist real numbers $\alpha_k, \beta_k$ 
($k = 0, 1, 2, \ldots$) depending on the weight function $w^2$
such that the polynomials $\tilde p_n(w^2, \cdot)$ satisfy the three-term recurrence relation
\begin{subequations}
\begin{equation}
	\label{eq:ttrr}
	\tilde p_{k+1}(w^2, t) = (t - \alpha_k) \tilde p_k(w^2, t) - \beta_k \tilde p_{k-1}(w^2, t)
	\qquad k = 0, 1, 2, \ldots
\end{equation}
with starting values
\begin{equation}
	\tilde p_0(w^2, t) = 1 
	\qquad \mbox{ and } \qquad
	\tilde p_{-1}(w^2, t) = 0,
\end{equation}
\end{subequations}
cf., e.g., \cite[eq.\ (1.3)]{Gau}. From \cite[Section 6.1]{Gau} we can then conclude that the desired zeros of the
polynomial $\tilde p_n(w^2, \cdot)$ (and hence also the zeros of $p_n(w^2, \cdot)$, i.e.\ the required interpolation nodes)
are the eigenvalues of the tridiagonal matrix 
\[
	M_n = \begin{pmatrix}
		\alpha_0 & \sqrt{\beta_1} & &  & & \\
		\sqrt{\beta_1} & \alpha_1  & \sqrt{\beta_2} & &  &  \\
		 & \ddots & \ddots & \ddots &  \\
		& &  \sqrt{\beta_{n-2}} & \alpha_{n-2}  & \sqrt{\beta_{n-1}} \\
		& & & \sqrt{\beta_{n-1}} & \alpha_{n-1} 
		\end{pmatrix}.
\]
So, to compute the interpolation nodes, we have to find the entries of the matrix $M_n$ and then calculate
its eigenvalues. 

For the former step, we must determine the values $\alpha_k$ and $\beta_k$. To this end, 
we first note that, in our case,
\[
	\alpha_k = 0 \qquad \mbox{ for all } k;
\]
this follows because the weight function $w^2$ that we have chosen is even.
For the $\beta_k$, we use the bootstrap method (also known
as the Stieltjes procedure) described in \cite[Section 4.1]{Gau} that can be formulated in the following way:
\begin{itemize}
\item Set $\beta_0 = \int_{-\infty}^\infty w^2(t) dt = \mu_0$ (which is known from eq.\ \eqref{eq:moments}).
\item For $k = 0, 1, 2, \ldots, n-1$:
	\begin{itemize}
	\item Compute $\tilde p_{k+1}(w^2,\cdot)$ by means of eq.\ \eqref{eq:ttrr}.
	\item Compute 
		\begin{equation}
			\label{eq:beta-stieltjes}
			\beta_{k+1} 
			= \frac{\int_{-\infty}^\infty w^2(t) \left( \tilde p_{k+1}(w^2, t) \right)^2 d t}
				{\int_{-\infty}^\infty w^2(t) \left( \tilde p_{k}(w^2, t) \right)^2 d t}.
		\end{equation}
	\end {itemize} 
\end{itemize}
Note that the computation of the integrals in eq.\ \eqref{eq:beta-stieltjes} is technically possible because
the coefficients of the polynomials $p_\ell(w^2, \cdot)$ ($\ell = k, k+1$) in the integrands have already 
been computed, so the squares of these polynomials can be computed as well, and therefore these integrals
can be expressed as linear combinations of the known moments $\mu_\ell$ with known coefficients.

It then remains to compute the eigenvalues of $M_n$.
In view of the tridiagonal structure of $M_n$, this is a straightforward process in numerical
linear algebra; the QR method, for example, is a reliable, stable and efficient method to accomplish this goal.

\subsection{Evaluation of $H_{w^2}[L_n; x]$} 

We now have all the components of the right-hand side of eq.\ \eqref{eq:lagrange} available, and so we can
compute the interpolation polynomial $L_n[f]$ and express it in the canonical form
\[
	L_n[f](t) = \sum_{k=0}^{n-1} \lambda_{kn}[f] t^k
\]
with certain coefficients $\lambda_{kn}[f]$ that depend on the given function $f$. It then follows that
\[
	H_{w^2}[L_n; x] 
	= \cpv_{-\infty}^\infty w^2(t) \frac{L_n[f](t)}{t-x} dt 
	= J_1[f](x) + L_n[f](x)  J_2(x)
\]
where
\[
	J_1[f](x) = \int_{-\infty}^\infty w^2(t) \frac{L_n[f](t) - L_n[f](x)}{t-x} dt
\]
and
\[
	J_2(x) = \cpv_{-\infty}^\infty w^2(t) \frac{1}{t-x} dt.
\]

For $J_1[f](x)$, we see for any $x \in \mathbb R$ that
\begin{eqnarray*}
	J_1[f](x) 
	&=& \int_{-\infty}^\infty w^2(t) 
		\frac{\sum_{k=0}^{n-1} \lambda_{kn}[f] t^k -  \sum_{k=0}^{n-1} \lambda_{kn}[f] x^k}{t-x} dt \\
	&=&  \int_{-\infty}^\infty w^2(t) \sum_{k=1}^{n-1} \lambda_{kn}[f] \frac{t^k - x^k}{t-x} dt \\
	&=&  \int_{-\infty}^\infty w^2(t) \sum_{k=1}^{n-1} \lambda_{kn}[f] 
							\sum_{\ell=0}^{k-1} t^\ell x^{k-\ell-1} dt \\
	&=& \sum_{k=1}^{n-1} \lambda_{kn}[f] \sum_{\ell=0}^{k-1} x^{k-\ell-1} \mu_\ell
\end{eqnarray*}
which can be evaluated with the help of eq.\ \eqref{eq:moments}.

For the integral $J_2(x)$, we first note that, owing to the fact that $w^2$ is an even function, $J_2$ is an odd function. 
Therefore, $J_2(0) = 0$ and $J_2(x) = - J_2(-x)$ for all $x < 0$. Hence it suffices to explicitly consider the computation
of $J_2(x)$ for $x > 0$; the remaining cases can be covered by symmetry arguments. In this case we can use the 
fact that our specific application uses $w^2(t) = \exp(-2t^4)$, cf.\ eq.\ \eqref{eq:wnum}, and argue as follows:
\[
	J_2(x) 
	= \int_{-\infty}^0 \frac{\exp(-2t^4)}{t-x} dt + \cpv_0^{\infty} \frac{\exp(-2t^4)}{t-x} dt 
	= J_{21}(x) + J_{22}(x),
\]
say, where (using the substitution $t = -u^{1/4}$ in the first step and the symbolic integration
capabilities of Mathematica \cite{Math} in the last one)
\begin{eqnarray*}
	J_{21}(x) 
	&=&  \int_{-\infty}^0 \frac{\exp(-2t^4)}{t-x} dt \\
	&=& \frac 1 4 \int_0^\infty \frac{\exp(-2u)}{-u - x u^{3/4}} du  \\
	&=& \frac 1 4 \int_0^\infty \exp(-2u) 
						\left[ \frac 1 {x^3 u^{1/4}} - \frac 1 {x^2 u^{1/2}} 
							+ \frac 1 {x u^{3/4}} - \frac 1 {x^3 (x + u^{1/4})} \right] du \\
	&=&   \frac{\Gamma(3/4)}{2^{11/4} x^3} 
		+ \frac{\Gamma(1/4)}{2^{9/4} x}  \\
	& & {} + \exp(-2x^4)
			\Bigg[
				\frac {\pi \mathrm i} 4 \mathop{\mathrm{erf}}(\mathrm i \sqrt{2} x^2)
				- \frac 1 4 \mathop{\mathrm{Ei}}(2 x^4) 
				- \frac {\pi \mathrm i}{2} \\
	& & \phantom{\exp(-2x^4) \Bigg] abc} {}
				+ \frac{3 \sqrt 2}{128} (-1 + \mathrm i) \Gamma \left( - \frac 1 4 \right) 
						\Gamma \left( - \frac 3 4, -2 x^4 \right)  \\
	& & \phantom{\exp(-2x^4) \Bigg] abc} {}
				- \frac{\sqrt 2}{8} (1 + \mathrm i)  \Gamma \left(\frac 5 4 \right) 
						\Gamma \left( - \frac 1 4, -2 x^4 \right) 
			\Bigg].
\end{eqnarray*}
In this formula, $\mathrm i$ is the imaginary unit, $\mathop{\mathrm{erf}}$ denotes the error function,
$\mathop{\mathrm{Ei}}$ is the exponential integral, and $\Gamma(\cdot, \cdot)$
is the incomplete Gamma function.
Note that, as one may expect since $J_{21}(x)$ is the integral of a real valued function over a real interval,
all the imaginary parts of the components of the final expression for $J_{21}(x)$ cancel each other, and so 
the result is purely real.

Finally, using the substitution $t = u^{1/4}$ in the first step and once again the symbolic integration
capabilities of Mathematica in the last one, we find
\begin{eqnarray*}
	J_{22}(x) \!\!\!\!\!
	&=&  \!\!\!\!\! \cpv_0^{\infty} \frac{\exp(-2t^4)}{t-x} dt \\
	&=&  \!\!\!\!\! \frac 1 4 \cpv_0^\infty \frac{\exp(-2u)}{u - x u^{3/4}} du  \\
	&=&  \!\!\!\!\! \frac 1 4 \cpv_0^\infty \exp(-2u) 
						\left[- \frac 1 {x^3 u^{1/4}} - \frac 1 {x^2 u^{1/2}} 
							- \frac 1 {x u^{3/4}} + \frac 1 {x^3 (u^{1/4} - x)} \right] du \\
	&=&  \!\!\!\!\!
		  - \frac{\sqrt 2 \pi} {8 x^2} - \frac{\Gamma(\nicefrac 3 4)}{2^{11/4} x^3} 
		- \frac{\Gamma(\nicefrac 1 4)}{2^{9/4} x}
		- \frac{\pi}{2x^4} G_{8,9}^{5,4} \!
			\left( \! \begin{matrix} \nicefrac 1 4, \nicefrac 1 2, \nicefrac 3 4, 1, \nicefrac 1 8, \nicefrac 3 8, \nicefrac 5 8, \nicefrac 7 8 \\
						\nicefrac 1 4, \nicefrac 1 2, \nicefrac 3 4, 1, 1, \nicefrac 1 8, \nicefrac 3 8, \nicefrac 5 8, \nicefrac 7 8
						\end{matrix}
			\, \Bigg| \, 2 x^4 \!\! \right)
\end{eqnarray*}
where $G_{8,9}^{5,4}$ is a member of the class of Meijer's $G$-functions (see, e.g., \cite{BS}).

Combining all the results listed in Appendix \ref{sec:num-comments}, we can implement the algorithm
used for computing the numerical results of Section \ref{sec:num-examples}.

\setcounter{equation}{0}
\section*{Acknowledgement}

The authors acknowledge the enormous contributions of Hee Sun Jung to the work in this paper. The authors thank Doron Lubinsky for his support for this work, Sheehan Olver for his helpful advice regarding some of the special functions occurring in the paper, 
and two anonymous referees for constructive and useful comments which helped improve the paper.

\setcounter{equation}{0}
\section*{Declaration of interest statement}
 
The authors report there are no competing interests to declare.

\end{document}